\documentclass[12pt]{amsart}
\usepackage{amsmath}
\usepackage{amssymb}
\usepackage{amsthm}
\usepackage{mathrsfs}
\usepackage{mathtools}
\usepackage{a4wide}
\newtheorem{prop}{Proposition}[section]
\newtheorem{lem}[prop]{Lemma}
\newtheorem{thm}[prop]{Theorem}
\newtheorem{conj}[prop]{Conjecture}
\numberwithin{equation}{section}
\allowdisplaybreaks[1]
\begin{document}
\title[On a conjecture on Schur functions]{On a conjecture on 2-reduced Schur functions and Schur's $Q$-functions}
\author[Y.~Nishiyama]{Yuta \textsc{Nishiyama}}
\address[Y.~Nishiyama]{Faculty~of~Advanced~Science~and~Technology, Kumamoto~University, Kumamoto~860-8555, Japan}
\email{ynishiyama@kumamoto-u.ac.jp}
\subjclass[2020]{Primary~05E05, Secondary~37K10, 05A10}
\keywords{2-reduced Schur function, Schur's $Q$-function, Littlewood--Richardson coefficient, inverse Kostka matrix, Korteweg--de~Vries equation.}
\date{}
\begin{abstract}
Motivated by Sato and Mori's work on the Korteweg--de~Vries (KdV) equation and the modified KdV equation, Mizukawa, Nakajima, and Yamada made a conjecture on $2$-reduced Schur functions and Schur's $Q$-functions.
The conjecture claims that certain sums of products of a Littlewood--Richardson coefficient and two $2$-reduced Schur functions are equal to Schur's $Q$-functions up to a scalar multiple.
In this paper we give a proof of the conjecture in cases which have not been proved yet.
We introduce a new expression of Schur's $Q$-functions and use it to prove the conjecture.
Combinatorics of the inverse Kostka matrix is also used.
We also provide consideration of the conjecture in general case.
\end{abstract}
\maketitle
\section*{Introduction}
The Korteweg--de~Vries (KdV) equation and the modified KdV equation are non-linear partial differential equations, and it is well-known that the equations have soliton solutions.
Sato (for example {\cite{S1981}}) developed the theory of such equations using the method of algebraic analysis (see {\cite{OSTT1988}}).
\par In {\cite{SM1980}}, Sato and Mori introduced some families of functions of infinitely many time variables to discuss behavior of the solutions of the equations (we refer to their statement in the end of Section \ref{section:Sato_and_Mori}).
These functions can be interpreted as symmetric functions (or symmetric polynomials) by an appropriate change of variables.
\par Nakajima~{\cite{N2019}} interpreted them as symmetric functions using $2$-reduced Schur functions $S_\lambda(x)$ and Littlewood--Richardson coefficients $c^\lambda_{\mu\nu}$.
This is a joint work with H.~Mizukawa and H.-F.~Yamada.
They also made a conjecture that these symmetric functions are equal to Schur's $Q$-functions $Q_\lambda(x,x)$ up to a scalar multiple (the definitions of the symbols are given in Section \ref{section:part_and_sym_poly}):
\begin{align*}
\sum_{\mu,\nu\in\mathscr P^{(n)}}c^\lambda_{\mu\nu}S_{2\mu+\Delta}(x)S_{2\nu+\Delta}(x)&=2^{-n}Q_{2\lambda+2\Delta}(x,x),\\
\sum_{\mu,\nu\in\mathscr P^{(n)}}c^\lambda_{\mu\nu}S_{2\mu+\delta}(x)S_{2\nu+\Delta}(x)&=2^{-n}Q_{2\lambda+\Delta+\delta}(x,x)
\end{align*}
for $n\geq1$ and $\lambda\in\mathscr P^{(n)}$, where $\mathscr P^{(n)}$ is the set of partitions of length at most $n$ and
\[\Delta\coloneqq(n,n-1,\ldots,2,1),\quad\delta\coloneqq(n-1,n-2,\ldots,1,0)\in\mathscr P^{(n)}.\]
Nakajima~{\cite{N2019}} proved the conjecture when $n=1$ and when $n=2$ under the assumption that $l(\lambda)\leq1$.
\par In this paper, we prove the conjecture when $n=2$ without any assumptions.
We use a new expression of Schur's $Q$-functions corresponding to partitions of length $2$ and combinatorics of the inverse Kostka matrix to prove the conjecture.
Consideration of the conjecture in the case of general $n$ is also provided, which reduces itself into a conjecture on Schur's $Q$-functions and the inverse Kostka matrix.
\par This paper consists as follows.
Firstly, we recall some basic concepts about partitions and symmetric polynomials in Section 1.
We refer to some arguments in {\cite{N2019}} and {\cite{SM1980}} about the conjecture in the next two sections.
In Section 2, we refer to the definition of the symmetric polynomials which is introduced by Sato and Mori.
We also refer to their statement about the solutions of the partial differential equations, which is stated by using the symmetric polynomials, in the section.
In Section 3, we refer to the conjecture which Mizukawa, Nakajima, and Yamada made.
We also refer to the proof in {\cite{N2019}} for the special cases of the conjecture.
We consider the conjecture in general case in Section 4, and we prove it when $n=2$ without any assumptions in Section 5.
\section{Partitions and symmetric polynomials}\label{section:part_and_sym_poly}
In this section, we recall some basic notations on partitions of integers and symmetric polynomials.
We mainly follow the notations used in {\cite{M1995}}.
\par A \textit{partition} is a weakly decreasing finite sequence of positive integers.
Let $\mathscr P$ be the set of partitions.
For $\lambda=(\lambda_1,\lambda_2,\ldots,\lambda_l)\in\mathscr P$ with $\lambda_1\geq\lambda_2\geq\cdots\geq\lambda_l>0$, denote $l(\lambda)\coloneqq l$ and call it the \textit{length} of $\lambda$.
Then we have $\lambda_n=0$ for $n>l$.
Write $\mathscr P^{(n)}\coloneqq\{\lambda\in\mathscr P\mid l(\lambda)\leq n\}$ and $\widetilde{\mathscr P}^{(n)}\coloneqq\{\lambda\in\mathscr P\mid\lambda_1\leq n\}$ for $n\geq0$.
For $\lambda\in\mathscr P$ and $i\geq1$, $m_i(\lambda)\coloneqq\#\{j\mid\lambda_j=i\}$ is called the \textit{multiplicity} of $i$ in $\lambda$.
Then $l(\lambda)=m_1(\lambda)+m_2(\lambda)+\cdots+m_n(\lambda)$ holds for $\lambda\in\widetilde{\mathscr P}^{(n)}$.
The partition $\lambda$ is uniquely determined by the sequence $(m_1(\lambda),m_2(\lambda),\ldots)$.
Hence a partition $\lambda\in\mathscr P$ is also written as $\lambda=(1^{m_1(\lambda)}2^{m_2(\lambda)}\cdots)$ by using the multiplicities.
For partitions $\lambda,\mu\in\mathscr P$, $\lambda\cup\mu\in\mathscr P$ is defined by $m_i(\lambda\cup\mu)\coloneqq m_i(\lambda)+m_i(\mu)$ for all $i$.
For $\lambda=(\lambda_1,\lambda_2,\ldots,\lambda_l)\in\mathscr P$, the partition $\lambda'$ is defined by $\lambda'_j\coloneqq\#\{i\mid\lambda_i\geq j\}$.
This $\lambda'$ is called the \textit{conjugate} of $\lambda$.
\par For $n\geq1$, the symmetric group $\mathfrak S_n$ acts on the polynomial ring $\mathbb Q[x_1,x_2,\ldots,x_n]$ by permuting the variables $x=(x_1,x_2,\ldots,x_n)$.
Precisely, the action of $\mathfrak S_n$ is defined by 
\[\sigma f(x_1,x_2,\ldots,x_n)\coloneqq f(x_{\sigma(1)},x_{\sigma(2)},\ldots,x_{\sigma(n)})\]
for $\sigma\in\mathfrak S_n$ and $f\in\mathbb Q[x_1,x_2,\ldots,x_n]$.
A polynomial in $\mathbb Q[x_1,x_2,\ldots,x_n]$ is said to be \textit{symmetric} if it is invariant under this action.
We denote the vector space consisting of symmetric polynomials in $\mathbb Q[x_1,x_2,\ldots,x_n]$ by $\Lambda_n$.
\par Write
\[x^\alpha\coloneqq{x_1}^{\alpha_1}{x_2}^{\alpha_2}\cdots{x_n}^{\alpha_n}\]
for a sequence of $n$ non-negative integers $\alpha=(\alpha_1,\alpha_2,\ldots,\alpha_n)\in(\mathbb Z_{\geq0})^n$.
For $\lambda=(\lambda_1,\lambda_2,\ldots,\lambda_l)\in\mathscr P^{(n)}$, let $m_\lambda(x)\coloneqq\sum_\alpha x^\alpha$, where $\alpha$ runs over distinct finite sequences $(\alpha_1,\alpha_2,\ldots,\alpha_n)$ obtained by permuting the parts of $(\lambda_1,\lambda_2,\ldots,\lambda_l,0,\ldots,0)\in(\mathbb Z_{\geq0})^n$.
This $m_\lambda(x)$ is called the \textit{monomial symmetric polynomial} corresponding to $\lambda$.
Then these $\{m_\lambda(x)\;|\;\lambda\in\mathscr P^{(n)}\}$ of monomial symmetric polynomials form a $\mathbb Q$-basis of $\Lambda_n$.
\par There are other well-known bases of $\Lambda_n$.
The \textit{elementary symmetric polynomial} $e_\lambda(x)$ is defined by
\[e_\lambda(x)\coloneqq m_{(1^{\lambda_1})}(x)m_{(1^{\lambda_2})}(x)\cdots m_{(1^{\lambda_{l(\lambda)}})}(x)\]
for $\lambda\in\widetilde{\mathscr P}^{(n)}$.
Then $\{e_\lambda(x)\;|\;\lambda\in\widetilde{\mathscr P}^{(n)}\}$ forms a $\mathbb Q$-basis of $\Lambda_n$ {\cite[Theorem 5.3.5]{P2015}}.
The \textit{power sum symmetric polynomial} $p_\lambda(x)$ is defined by
\[p_\lambda(x)\coloneqq m_{(\lambda_1)}(x)m_{(\lambda_2)}(x)\cdots m_{(\lambda_{l(\lambda)})}(x)\]
for $\lambda\in\mathscr P$.
Then $\{p_\lambda(x)\;|\;\lambda\in\mathscr P^{(n)}\}$ forms a $\mathbb Q$-basis of $\Lambda_n$ {\cite[Theorem 5.3.9]{P2015}}.
\par Define the antisymmetric polynomial $a_\alpha(x)$ by
\[a_\alpha(x)\coloneqq\sum_{\sigma\in\mathfrak S_n}\mathrm{sgn}(\sigma)\sigma(x^\alpha)\]
for $\alpha=(\alpha_1,\alpha_2,\ldots,\alpha_n)\in(\mathbb Z_{\geq0})^n$, where $\mathrm{sgn}(\sigma)=\pm1$ is the sign of the permutation $\sigma$.
Write $\delta\coloneqq(n-1,n-2,\ldots,1,0)\in\mathscr P^{(n)}$.
Since $a_{\lambda+\delta}(x)$ is divisible by $a_\delta(x)$ for $\lambda\in\mathscr P^{(n)}$, the \textit{Schur polynomial} $s_\lambda(x)\in\mathbb Q[x_1,x_2,\ldots,x_n]$ can be defined by
\[s_\lambda(x)\coloneqq\dfrac{a_{\lambda+\delta}(x)}{a_\delta(x)}\]
for $\lambda\in\mathscr P^{(n)}$.
Then $\{s_\lambda(x)\;|\;\lambda\in\mathscr P^{(n)}\}$ forms a $\mathbb Q$-basis of $\Lambda_n$ {\cite[Theorem 5.4.4]{P2015}}.
\par Since $s_\mu(x)s_\nu(x)\in\Lambda_n$ for $\lambda,\mu\in\mathscr P^{(n)}$, it can be written as a linear combination of Schur polynomials.
We write the coefficients $c^\lambda_{\mu\nu}$:
\[s_\mu(x)s_\nu(x)=\sum_{\lambda\in\mathscr P^{(n)}}c^\lambda_{\mu\nu}s_\lambda(x).\]
This coefficient $c^\lambda_{\mu\nu}$ is called the \textit{Littlewood--Richardson coefficient}.
\par The \textit{Kostka matrix} $(K_{\lambda\mu})_{\lambda,\mu\in\mathscr P^{(n)}}$ is defined as the coefficients that arise when one expresses a Schur polynomial as a linear combination of monomial symmetric polynomials:
\[s_\lambda(x)=\sum_{\mu\in\mathscr P^{(n)}}K_{\lambda\mu}m_\mu(x)\]
for $\lambda\in\mathscr P^{(n)}$.
The entries of the inverse matrix of the Kostka matrix are written as $K^{-1}_{\lambda\mu}$, i.e.,
\[\sum_{\nu\in\mathscr P^{(n)}}K_{\lambda\nu}K^{-1}_{\nu\mu}=\sum_{\nu\in\mathscr P^{(n)}}K^{-1}_{\lambda\nu}K_{\nu\mu}=\delta_{\lambda\mu}\]
for $\lambda,\mu\in\mathscr P^{(n)}$, where $\delta_{\lambda\mu}$ is the Kronecker delta.
Using the inverse Kostka matrix, one can express a Schur polynomial as a linear combination of the elementary symmetric polynomials:
\begin{prop}[see {\cite[Chapter I, p.101]{M1995}}]\label{prop:M101}
For $\lambda\in\mathscr P^{(n)}$, it holds that
\[s_\lambda(x)=\sum_{\mu\in\widetilde{\mathscr P}^{(n)}}K^{-1}_{\mu\lambda'}e_\mu(x).\]
\end{prop}
\par One can consider the compound of the variables: for two sets of variables
\[x=(x_1,x_2,\ldots,x_n), \quad y=(y_1,y_2,\ldots,y_n)\]
and a symmetric polynomial $f(x,y)$ in variables $x\cup y=(x_1,x_2,\ldots,x_n,y_1,y_2,\ldots,y_n)$, one obtains a symmetric polynomial $f(x,x)$ in variables $x$ by putting $x_i=y_i$ for all $1\leq i\leq n$.
\par The sequence of symmetric polynomials $\{q_r(x)\mid r\geq0\}$ is defined by
\[\sum_{r=0}^\infty q_r(x)t^r\coloneqq\prod_{i=1}^n\dfrac{1+x_it}{1-x_it}=\prod_{i=1}^n(1+x_it)(1+x_it+x_i^2t^2+\cdots).\]
We write $q_r(x)\coloneqq0$ if $r<0$.
Define the \textit{$2$-reduced Schur polynomials} $S_\lambda(x)$ (see~{\cite{ANY1999}}) using $q_r(x)$ as
\[S_\lambda(x)\coloneqq\det(q_{\lambda_i-i+j}(x))_{1\leq i,j\leq n}\]
for $\lambda\in\mathscr P^{(n)}$.
The Schur polynomials and the $2$-reduced Schur polynomials satisfy the following proposition:
\begin{prop}[see {\cite[Chapter III, (4.7)]{M1995}}]\label{prop:M225(4.7)}
Let $x=(x_1,x_2,\ldots,x_n)$, $y=(y_1,y_2,\ldots,y_n)$ be two sets of $n$ variables.
Then it holds that
\[\sum_{\lambda\in\mathscr P^{(n)}}s_\lambda(y)S_\lambda(x)=\prod_{i=1}^n\prod_{j=1}^n\dfrac{1+x_iy_j}{1-x_iy_j}.\]
\end{prop}
\par\textit{Schur's $Q$-polynomials} $Q_\lambda(x)$, also known as \textit{Schur's $Q$-functions}, are also defined by using $q_r(x)$.
For $r,s\geq0$, $Q_{(r,s)}(x)$ is defined by
\[Q_{(r,s)}(x)\coloneqq
\begin{cases}
q_r(x)q_s(x)+2\displaystyle{\sum_{i=1}^s}(-1)^iq_{r+i}(x)q_{s-i}(x)&((r,s)\ne(0,0))\\
0&((r,s)=(0,0)).
\end{cases}\]
For $\lambda=(\lambda_1,\lambda_2,\ldots,\lambda_{2n})\in\mathscr P^{(2n)}$ ($\lambda_1\geq\lambda_2\geq\cdots\geq\lambda_{2n-1}>0$ and $\lambda_{2n}$ may be $0$), $Q_\lambda(x)$ is defined by using $Q_{(r,s)}(x)$.
Since $Q_{(r,s)}(x)=-Q_{(s,r)}(x)$ holds for all $r$ and $s$, the $2n\times2n$ matrix $(Q_{(\lambda_i,\lambda_j)}(x))_{1\leq i,j\leq2n}$ is skew-symmetric and $Q_\lambda(x)$ is defined as the Pfaffian of the matrix:
\[Q_\lambda(x)\coloneqq\mathrm{Pf}((Q_{(\lambda_i,\lambda_j)}(x))_{1\leq i,j\leq2n}).\]
By compounding the variables, $\{q_r(x,x)\mid r\geq0\}$ is defined by
\[\sum_{r=0}^\infty q_r(x,x)t^r\coloneqq\prod_{i=1}^n\left(\dfrac{1+x_it}{1-x_it}\right)^2,\]
and Schur's $Q$-polynomials $Q_\lambda(x,x)$ are defined by using $\{q_r(x,x)\mid r\geq0\}$ similarly.
\section{Sato and Mori's symmetric polynomials}\label{section:Sato_and_Mori}
Sato and Mori~{\cite{SM1980}} introduced the symmetric polynomials (we write them as $f_\lambda(x)$ and $g_\lambda(x)$ in the present paper) to study the KdV (and modified KdV) equation.
In this section we define the polynomials following {\cite{SM1980}}, and express them by $2$-reduced Schur polynomials.
The method and some notations used in the section (including $f_\lambda(x)$ and $g_\lambda(x)$) are due to {\cite{N2019}}.
\par For a function $f(y)=f(y_1,y_2,\ldots,y_n)$ of variables $y=(y_1,y_2,\ldots,y_n)$, define its \textit{totally even part} $\mathrm{TE}_y(f(y))$ and \textit{totally odd part} $\mathrm{TO}_y(f(y))$ as follows:
\begin{align*}
&\mathrm{TE}_y(f(y_1,y_2,\ldots,y_n))\coloneqq2^{-n}\sum_{\epsilon_1,\epsilon_2,\ldots,\epsilon_n\in\{\pm1\}}f(\epsilon_1y_1,\epsilon_2y_2,\ldots,\epsilon_ny_n),\\
&\mathrm{TO}_y(f(y_1,y_2,\ldots,y_n))\coloneqq2^{-n}\sum_{\epsilon_1,\epsilon_2,\ldots,\epsilon_n\in\{\pm1\}}(-1)^{\alpha(\epsilon)}f(\epsilon_1y_1,\epsilon_2y_2,\ldots,\epsilon_ny_n),
\end{align*}
where $\alpha(\epsilon)\coloneqq\#\{1\leq i\leq n\mid\epsilon_i=-1\}$.
When $f(y)$ is a polynomial in $y_1,y_2,\ldots,y_n$ and written as the form
\[f(y)=\sum_{a_1,a_2,\ldots,a_n\geq0}c_{a_1,a_2,\ldots,a_n}y_1^{a_1}y_2^{a_2}\cdots y_n^{a_n}\]
with scalar $c_{a_1,a_2,\ldots,a_n}$, the equalities
\begin{align*}
\mathrm{TE}_y(f(y))&=\sum_{\substack{a_1,a_2,\ldots,a_n\geq0\\a_1,a_2,\ldots,a_n\text{:even}}}c_{a_1,a_2,\ldots,a_n}y_1^{a_1}y_2^{a_2}\cdots y_n^{a_n},\\
\mathrm{TO}_y(f(y))&=\sum_{\substack{a_1,a_2,\ldots,a_n\geq0\\a_1,a_2,\ldots,a_n\text{:odd}}}c_{a_1,a_2,\ldots,a_n}y_1^{a_1}y_2^{a_2}\cdots y_n^{a_n}
\end{align*}
hold.
\par Define functions $\Phi^\pm(x,y)$ of variables $(x,y)=(x_1,x_2,\ldots,x_n,y_1,y_2,\ldots,y_n)$ as follows:
\begin{align*}
\Phi^+(x,y)&\coloneqq\dfrac{1}{a_\delta(y^2)}\mathrm{TE}_y\left(a_\delta(y)\exp\left(\sum_{\substack{m\geq1\\m\text{:odd}}}\dfrac{2}{m}p_{(m)}(x)p_{(m)}(y)\right)\right),\\
\Phi^-(x,y)&\coloneqq\dfrac{1}{a_\delta(y^2)}\mathrm{TO}_y\left(a_\delta(y)\exp\left(\sum_{\substack{m\geq1\\m\text{:odd}}}\dfrac{2}{m}p_{(m)}(x)p_{(m)}(y)\right)\right),
\end{align*}
where $y^2\coloneqq(y_1^2,y_2^2,\ldots,y_n^2)$.
These $\Phi^\pm(x,y)$ satisfy the following proposition:
\begin{prop}[\cite{N2019}]\label{prop:Phi^+-}
The equalities
\begin{align*}
\Phi^+(x,y)&=\sum_{\lambda\in\mathscr P^{(n)}}s_\lambda(y^2)S_{2\lambda+\delta}(x),\\
\Phi^-(x,y)&=\sum_{\lambda\in\mathscr P^{(n)}}y_1y_2\cdots y_ns_\lambda(y^2)S_{2\lambda+\Delta}(x)
\end{align*}
hold with $\Delta\coloneqq(n,n-1,\ldots,2,1)$, $\delta\coloneqq(n-1,n-2,\ldots,1,0)\in\mathscr P^{(n)}$, where
$2\lambda\coloneqq(2\lambda_1,2\lambda_2,\ldots,2\lambda_n)$ and $\lambda+\mu\coloneqq(\lambda_1+\mu_1,\lambda_2+\mu_2,\ldots,\lambda_n+\mu_n)$ for $\lambda,\mu\in\mathscr P^{(n)}$.
\end{prop}
\begin{proof}
Since
\begin{align*}
&\log\prod_{i=1}^n\prod_{j=1}^n\dfrac{1+x_iy_j}{1-x_iy_j}=\sum_{i=1}^n\sum_{j=1}^n(\log(1+x_iy_j)-\log(1-x_iy_j))\\
&=\sum_{i=1}^n\sum_{j=1}^n\left(\sum_{m\geq1}\dfrac{(-1)^{m+1}}{m}(x_iy_i)^m+\sum_{m\geq1}\dfrac{1}{m}(x_iy_j)^m\right)\\
&=\sum_{i=1}^n\sum_{j=1}^n\sum_{m\geq1}\dfrac{(-1)^{m+1}+1}{m}(x_iy_j)^m\\
&=\sum_{i=1}^n\sum_{j=1}^n\sum_{\substack{m\geq1\\m\text{:odd}}}\dfrac{2}{m}(x_iy_j)^m
=\sum_{\substack{m\geq1\\m\text{:odd}}}\dfrac{2}{m}p_{(m)}(x)p_{(m)}(y),
\end{align*}
we have
\begin{align*}
&a_\delta(y)\exp\left(\sum_{\substack{m\geq1\\m\text{:odd}}}\dfrac{2}{m}p_{(m)}(x)p_{(m)}(y)\right)
=a_\delta(y)\prod_{i=1}^n\prod_{j=1}^n\dfrac{1+x_iy_j}{1-x_iy_j}
=a_\delta(y)\sum_{\mu\in\mathscr P^{(n)}}s_\mu(y)S_\mu(x)\\
&=\sum_{\mu\in\mathscr P^{(n)}}a_{\mu+\delta}(y)S_\mu(x)
\end{align*}
from Proposition \ref{prop:M225(4.7)}.
Therefore,
\begin{align*}
\Phi^+(x,y)
&=\dfrac{1}{a_\delta(y^2)}\mathrm{TE}_y\left(a_\delta(y)\exp\left(\sum_{\substack{m\geq1\\m\text{:odd}}}\dfrac{2}{m}p_{(m)}(x)p_{(m)}(y)\right)\right)\\
&=\dfrac{1}{a_\delta(y^2)}\mathrm{TE}_y\left(\sum_{\mu\in\mathscr P^{(n)}}a_{\mu+\delta}(y)S_\mu(x)\right)
\end{align*}
hold.
Partitions $\mu$ corresponding to the remaining summands can be written as $2\lambda+\delta$ by $\lambda\in\mathscr P^{(n)}$, and we obtain
\begin{align*}
&\Phi^+(x,y)
=\dfrac{1}{a_\delta(y^2)}\sum_{\lambda\in\mathscr P^{(n)}}a_{2\lambda+2\delta}(y)S_{2\lambda+\delta}(x)
=\dfrac{1}{a_\delta(y^2)}\sum_{\lambda\in\mathscr P^{(n)}}a_{\lambda+\delta}(y^2)S_{2\lambda+\delta}(x)\\
&=\sum_{\lambda\in\mathscr P^{(n)}}s_\lambda(y^2)S_{2\lambda+\delta}(x)
\end{align*}
as desired.
The second formula can be proved similarly.
\end{proof}
\par The polynomial $a_\delta(y^2)$ is skew-symmetric for the variables $y$, and $\Phi^\pm(x,y)$ are symmetric both for $x$, and for $y$.
Hence there uniquely exist families of symmetric polynomials $\{f_\lambda(x)\mid\lambda\in\mathscr P^{(n)}\}$ and $\{g_\lambda(x)\mid\lambda\in\mathscr P^{(n)}\}$ satisfying the following:
\begin{align*}
a_\delta(y^2)(\Phi^-(x,y))^2&=\sum_{\lambda\in\mathscr P^{(n)}}a_{\lambda+\Delta}(y^2)f_\lambda(x),\\
a_\delta(y^2)\Phi^+(x,y)\Phi^-(x,y)&=\sum_{\lambda\in\mathscr P^{(n)}}y_1y_2\cdots y_na_{\lambda+\Delta}(y^2)g_\lambda(x).
\end{align*}
These families are expressed by $2$-reduced Schur polynomials.
\begin{prop}[\cite{N2019}]\label{prop:f,g=cSS}
The equalities
\begin{align*}
f_\lambda(x)&=\sum_{\mu,\nu\in\mathscr P^{(n)}}c^\lambda_{\mu\nu}S_{2\mu+\Delta}(x)S_{2\nu+\Delta}(x),\\
g_\lambda(x)&=\sum_{\mu,\nu\in\mathscr P^{(n)}}c^\lambda_{\mu\nu}S_{2\mu+\delta}(x)S_{2\nu+\Delta}(x)
\end{align*}
hold for $\lambda\in\mathscr P^{(n)}$.
\end{prop}
\begin{proof}
Using Proposition \ref{prop:Phi^+-} one obtains
\begin{align*}
a_\delta(y^2)(\Phi^-(x,y))^2
&=a_\delta(y^2)\left(\sum_{\mu\in\mathscr P^{(n)}}y_1y_2\cdots y_ns_\mu(y^2)S_{2\mu+\Delta}(x)\right)^2\\
&=a_\delta(y^2)\sum_{\mu,\nu\in\mathscr P^{(n)}}y_1^2y_2^2\cdots y_n^2s_\mu(y^2)s_\nu(y^2)S_{2\mu+\Delta}(x)S_{2\nu+\Delta}(x)\\
&=a_\delta(y^2)\sum_{\lambda,\mu,\nu\in\mathscr P^{(n)}}y_1^2y_2^2\cdots y_n^2c^\lambda_{\mu\nu}s_\lambda(y^2)S_{2\mu+\Delta}(x)S_{2\nu+\Delta}(x)\\
&=\sum_{\lambda,\mu,\nu\in\mathscr P^{(n)}}y_1^2y_2^2\cdots y_n^2c^\lambda_{\mu\nu}a_{\lambda+\delta}(y^2)S_{2\mu+\Delta}(x)S_{2\nu+\Delta}(x)\\
&=\sum_{\lambda,\mu,\nu\in\mathscr P^{(n)}}c^\lambda_{\mu\nu}a_{\lambda+\Delta}(y^2)S_{2\mu+\Delta}(x)S_{2\nu+\Delta}(x)\\
&=\sum_{\lambda\in\mathscr P^{(n)}}a_{\lambda+\Delta}(y^2)\sum_{\mu,\nu\in\mathscr P^{(n)}}c^\lambda_{\mu\nu}S_{2\mu+\Delta}(x)S_{2\nu+\Delta}(x),
\end{align*}
and therefore,
\[f_\lambda(x)=\sum_{\mu,\nu\in\mathscr P^{(n)}}c^\lambda_{\mu\nu}S_{2\mu+\Delta}(x)S_{2\nu+\Delta}(x)\]
holds from the definition of $f_\lambda(x)$.
The expression of $g_\lambda(x)$ is also derived in a similar way.
\end{proof}
The families $\{f_\lambda(x)\mid\lambda\in\mathscr P^{(n)}\}$ and $\{g_\lambda(x)\mid\lambda\in\mathscr P^{(n)}\}$ are introduced by Sato and Mori~{\cite{SM1980}} originally (in some different notation) to state about behavior of the solutions of the KdV (and modified KdV) equation.
They stated the following as a theorem without any proof:
\begin{thm}[\cite{SM1980}]\label{thm:Sato_and_Mori}
Let $u=u(s)$ be a function of a countably infinite number of variables $s=(s_0,s_1,\ldots)$ which satisfies the KdV equation
\[u_{s_1}=(u_{s_0s_0}+3u^2)_{s_0},\]
and $v=v(s)$ be a function of the same variables $s$ which satisfies the modified KdV equation
\[v_{s_1}=(v_{s_0s_0}-2v^3)_{s_0}.\]
Here the subscripts represent the derivatives.
Let $f=f(s)$ and $g=g(s)$ be functions of $s$ which satisfy
\[u=(2\log f)_{s_0s_0}\quad\text{and}\quad v=\left(\log\dfrac{f}{g}\right)_{s_0}.\]
For $\lambda\in\mathscr P^{(n)}$, define functions $\tilde f_\lambda(t)$ and $\tilde g_\lambda(t)$ of variables $t=(t_0,t_1,\ldots)$ by changing the variables of $f_\lambda(x)$ and $g_\lambda(x)$ respectively, as
\[p_{(2m+1)}(x)\mapsto\dfrac{2m+1}{4}t_m.\]
Then there exist families
\[\bigcup_{n=0}^\infty\{K_{2\lambda+2\Delta_n}(u)\mid\lambda\in\mathscr P^{(n)}\}\quad\text{and}\quad\bigcup_{n=0}^\infty\{K_{2\lambda+\Delta_n+\delta_n}(u)\mid\lambda\in\mathscr P^{(n)}\}\]
of polynomials in $u,u_{s_0},u_{s_0s_0},\ldots$, where
\[\delta_n\coloneqq(n-1,n-2,\ldots,1,0)\quad\text{and}\quad\Delta_n\coloneqq(n,n-1,\ldots,2,1)\in\mathscr P^{(n)},\]
such that
\begin{align*}
\dfrac{f(s+t)f(s-t)}{f(s)^2}&=\sum_{n=0}^\infty\sum_{\lambda\in\mathcal P^{(n)}}2^n\tilde f_\lambda(t)K_{2\lambda+2\Delta_n}(u),\\
\dfrac{f(s+t)g(s-t)}{f(s)g(s)}&=\sum_{n=0}^\infty\sum_{\lambda\in\mathcal P^{(n)}}2^n\tilde g_\lambda(t)K_{2\lambda+\Delta_n+\delta_n}(u)
\end{align*}
hold.
\end{thm}
\section{A conjecture on $2$-reduced Schur functions and Schur's $Q$-functions}
Mizukawa, Nakajima, and Yamada~{\cite{N2019}} made the following conjecture that claims that $f_\lambda(x)$ and $g_\lambda(x)$ in the previous section coincide with Schur's $Q$-polynomials $Q_\lambda(x,x)$:
\begin{conj}[\cite{N2019}]\label{conj:cSS=Q}
The equalities
\begin{align*}
\sum_{\mu,\nu\in\mathscr P^{(n)}}c^\lambda_{\mu\nu}S_{2\mu+\Delta}(x)S_{2\nu+\Delta}(x)&=2^{-n}Q_{2\lambda+2\Delta}(x,x),\\
\sum_{\mu,\nu\in\mathscr P^{(n)}}c^\lambda_{\mu\nu}S_{2\mu+\delta}(x)S_{2\nu+\Delta}(x)&=2^{-n}Q_{2\lambda+\Delta+\delta}(x,x)
\end{align*}
hold for $n\geq1$ and $\lambda\in\mathscr P^{(n)}$.
\end{conj}
The former formula of Conjecture \ref{conj:cSS=Q} corresponds to $f_\lambda(x)$, and the latter one corresponds to $g_\lambda(x)$.
Nakajima~{\cite{N2019}} proved the conjecture in the special cases that $n=1$, and that $n=2$ and $l(\lambda)\leq1$.
In this section we give the proof for these special cases following him.
First, we give a lemma on $q_r(x)$:
\begin{lem}[\cite{N2019}]\label{lem:q(x,x)=2sumq(x)q(x)}
For $u\geq0$, the following identities hold:
\begin{align}
\label{q_2u(x,x)=2sumq_2k(x)q_2k(x)}
q_{2(u+1)}(x,x)&=2\sum_{k=0}^{u+1}q_{2k}(x)q_{2(u+1-k)}(x),\\
\label{q_2u(x,x)=2sumq_2k+1(x)q_2k+1(x)}
q_{2(u+1)}(x,x)&=2\sum_{k=0}^uq_{2k+1}(x)q_{2(u-k)+1}(x),\\
\label{q_2u+1(x,x)=2sumq_2k+1(x)q_2k(x)}
q_{2u+1}(x,x)&=2\sum_{k=0}^uq_{2k+1}(x)q_{2(u-k)}(x).
\end{align}
\end{lem}
\begin{proof}
Since
\begin{align*}
&\sum_{r=0}^\infty q_r(x,x)t^r
=\prod_{i=1}^n\left(\dfrac{1+x_it}{1-x_it}\right)^2
=\left(\sum_{r_1=0}^\infty q_{r_1}(x)t^{r_1}\right)\left(\sum_{r_2=0}^\infty q_{r_2}(x)t^{r_2}\right)\\
&=\sum_{r=0}^\infty\sum_{s=0}^rq_s(x)q_{r-s}(x)t^r
\end{align*}
from the definition,
\begin{equation}\label{q(x,x)}
q_r(x,x)=\sum_{s=0}^rq_s(x)q_{r-s}(x)
\end{equation}
holds for $r\geq0$.
On the other hand,
\begin{align*}
&\sum_{r=0}^\infty\left(\sum_{s=0}^r(-1)^sq_s(x)q_{r-s}(x)\right)t^r
=\left(\sum_{r_1=0}^\infty q_{r_1}(x)t^{r_1}\right)\left(\sum_{r_2=0}^\infty q_{r_2}(x)(-t)^{r_2}\right)\\
&=\left(\prod_{i=1}^n\dfrac{1+x_it}{1-x_it}\right)\left(\prod_{i=1}^n\dfrac{1+x_i(-t)}{1-x_i(-t)}\right)
=1
\end{align*}
and
\[\sum_{s=0}^r(-1)^sq_s(x)q_{r-s}(x)=0\]
holds for $r\geq1$.
Therefore,
\begin{align*}
&q_{2(u+1)}(x,x)
=\sum_{s=0}^{2(u+1)}q_s(x)q_{2(u+1)-s}(x)\\
&=\sum_{s=0}^{2(u+1)}q_s(x)q_{2(u+1)-s}(x)+\sum_{s=0}^{2(u+1)}(-1)^sq_s(x)q_{2(u+1)-s}(x)\\
&=2\sum_{k=0}^{u+1}q_{2k}(x)q_{2(u+1-k)}(x)
\end{align*}
and (\ref{q_2u(x,x)=2sumq_2k(x)q_2k(x)}) holds.
Remaining identities (\ref{q_2u(x,x)=2sumq_2k+1(x)q_2k+1(x)}) and (\ref{q_2u+1(x,x)=2sumq_2k+1(x)q_2k(x)}) can be proved similarly.
\end{proof}
Now we prove the conjecture in the special cases following {\cite{N2019}}.
\begin{thm}[\cite{N2019}]
Conjecture \ref{conj:cSS=Q} holds for $n=1$;
\begin{align*}
\sum_{\mu,\nu\in\mathscr P^{(1)}}c^\lambda_{\mu\nu}S_{2\mu+\Delta}(x)S_{2\nu+\Delta}(x)&=2^{-n}Q_{2\lambda+2\Delta}(x,x),\\
\sum_{\mu,\nu\in\mathscr P^{(1)}}c^\lambda_{\mu\nu}S_{2\mu+\delta}(x)S_{2\nu+\Delta}(x)&=2^{-n}Q_{2\lambda+\Delta+\delta}(x,x)
\end{align*}
hold for $\lambda\in\mathscr P^{(1)}$, where $\Delta=(1), \delta=(0)\in\mathscr P^{(1)}$.
\end{thm}
\begin{proof}
Let $\lambda=(\lambda_1)$.
The Littlewood--Richardson rule (see e.g. {\cite[Chapter 7, A1.3.3]{S1999}}) tells us that $c^\lambda_{\mu\nu}=1$ for $(\mu,\nu)=((k),(\lambda_1-k))$ with $0\leq k\leq\lambda_1$, and $c^\lambda_{\mu\nu}=0$ otherwise.
Therefore, the formulae of the theorem can be rewritten as follows:
\begin{align*}
\sum_{k=0}^{\lambda_1}q_{2k+1}(x)q_{2(\lambda_1-k)+1}(x)&=2^{-1}q_{2\lambda_1+2}(x,x),\\
\sum_{k=0}^{\lambda_1}q_{2k}(x)q_{2(\lambda_1-k)+1}(x)&=2^{-1}q_{2\lambda_1+1}(x,x).
\end{align*}
They hold from Lemma \ref{lem:q(x,x)=2sumq(x)q(x)} and it completes the proof.
\end{proof}
\begin{thm}[\cite{N2019}]
Conjecture \ref{conj:cSS=Q} holds for the case that $n=2$ and $l(\lambda)\leq1$;
\begin{align*}
\sum_{\mu,\nu\in\mathscr P^{(2)}}c^\lambda_{\mu\nu}S_{2\mu+\Delta}(x)S_{2\nu+\Delta}(x)&=2^{-n}Q_{2\lambda+2\Delta}(x,x),\\
\sum_{\mu,\nu\in\mathscr P^{(2)}}c^\lambda_{\mu\nu}S_{2\mu+\delta}(x)S_{2\nu+\Delta}(x)&=2^{-n}Q_{2\lambda+\Delta+\delta}(x,x)
\end{align*}
hold for $\lambda\in\mathscr P^{(1)}$, where $\Delta\coloneqq(2,1), \delta\coloneqq(1,0)\in\mathscr P^{(2)}$.
\end{thm}
\begin{proof}
Let $\lambda=(\lambda_1)$.
The formulae of the theorem are rewritten as follows using the Littlewood--Richardson rule:
\begin{align*}
\sum_{k=0}^{\lambda_1}S_{(2k+2,1)}(x)S_{(2(\lambda_1-k)+2,1)}(x)&=2^{-2}Q_{(2\lambda_1+4,2)}(x,x),\\
\sum_{k=0}^{\lambda_1}S_{(2k+1)}(x)S_{(2(\lambda_1-k)+2,1)}(x)&=2^{-2}Q_{(2\lambda_1+3,1)}(x,x).
\end{align*}
They can be rewritten using $q_r(x)$ and $q_r(x,x)$:
\begin{align}\label{cSS=Q,even,n=2,l(lambda)=1}
\begin{split}
&\sum_{k=0}^{\lambda_1}
\det\begin{pmatrix}
q_{2k+2}(x)&q_{2k+3}(x)\\
1&q_1(x)
\end{pmatrix}
\det\begin{pmatrix}
q_{2(\lambda_1-k)+2}(x)&q_{2(\lambda_1-k)+3}(x)\\
1&q_1(x)
\end{pmatrix}\\
&=2^{-2}(q_{2\lambda_1+4}(x,x)q_2(x,x)-2q_{2\lambda_1+5}(x,x)q_1(x,x)+2q_{2\lambda_1+6}(x,x)),
\end{split}
\end{align}
\begin{align}\label{cSS=Q,odd,n=2,l(lambda)=1}
\begin{split}
&\sum_{k=0}^{\lambda_1}q_{2k+1}(x)
\det\begin{pmatrix}
q_{2(\lambda_1-k)+2}(x)&q_{2(\lambda_1-k)+3}(x)\\
1&q_1(x)
\end{pmatrix}\\
&=2^{-2}(q_{2\lambda_1+3}(x,x)q_1(x,x)-2q_{2\lambda_1+4}(x,x)).
\end{split}
\end{align}
The left-hand side of (\ref{cSS=Q,even,n=2,l(lambda)=1}) is
\begin{align*}
&\sum_{k=0}^{\lambda_1}
\det\begin{pmatrix}
q_{2k+2}(x)&q_{2k+3}(x)\\
1&q_1(x)
\end{pmatrix}
\det\begin{pmatrix}
q_{2(\lambda_1-k)+2}(x)&q_{2(\lambda_1-k)+3}(x)\\
1&q_1(x)
\end{pmatrix}\\
&=\sum_{k=0}^{\lambda_1}(q_{2k+2}(x)q_1(x)-q_{2k+3}(x))(q_{2(\lambda_1-k)+2}(x)q_1(x)-q_{2(\lambda_1-k)+3}(x))\\
&=\sum_{k=0}^{\lambda_1}(q_{2k+2}(x)q_{2(\lambda_1-k)+2}(x)q_1(x)^2-q_{2k+2}(x)q_{2(\lambda_1-k)+3}(x)q_1(x)\\
&\qquad\qquad-q_{2k+3}(x)q_{2(\lambda_1-k)+2}(x)q_1(x)+q_{2k+3}(x)q_{2(\lambda_1-k)+3}(x))\\
&=\sum_{k=0}^{\lambda_1}(q_{2k+2}(x)q_{2(\lambda_1-k)+2}(x)q_1(x)^2-2q_{2k+2}(x)q_{2(\lambda_1-k)+3}(x)q_1(x)\\
&\qquad\qquad+q_{2k+3}(x)q_{2(\lambda_1-k)+3}(x))
\end{align*}
and each term can be calculated as follows using Lemma \ref{lem:q(x,x)=2sumq(x)q(x)}:
\begin{align*}
&\sum_{k=0}^{\lambda_1}q_{2k+2}(x)q_{2(\lambda_1-k)+2}(x)q_1(x)^2\\
&=(2^{-1}q_{2\lambda_1+4}(x,x)-2q_{2\lambda_1+4}(x)q_0(x))q_1(x)^2\\
&=2^{-1}q_{2\lambda_1+4}(x,x)q_1(x)^2-2q_{2\lambda_1+4}(x)q_0(x)q_1(x)^2\\
&=2^{-2}q_{2\lambda_1+4}(x,x)q_2(x,x)-q_{2\lambda_1+4}(x)q_2(x,x),
\end{align*}
\begin{align*}
&-2\sum_{k=0}^{\lambda_1}q_{2k+2}(x)q_{2(\lambda_1-k)+3}(x)q_1(x)\\
&=-(q_{2\lambda_1+5}(x,x)-2q_0(x)q_{2\lambda_1+5}(x)-2q_{2\lambda_1+4}(x)q_1(x))q_1(x)\\
&=-q_{2\lambda_1+5}(x,x)q_1(x)+2q_0(x)q_{2\lambda_1+5}(x)q_1(x)+2q_{2\lambda_1+4}(x)q_1(x)^2\\
&=-2^{-1}q_{2\lambda_1+5}(x,x)q_1(x,x)+2q_{2\lambda_1+5}(x)q_1(x)+q_{2\lambda_1+4}(x)q_2(x,x),
\end{align*}
\[\sum_{k=0}^{\lambda_1}q_{2k+3}(x)q_{2(\lambda_1-k)+3}(x)=2^{-1}q_{2\lambda_1+6}(x,x)-2q_1(x)q_{2\lambda_1+5}(x).\]
Therefore, the left-hand side of (\ref{cSS=Q,even,n=2,l(lambda)=1}) is equals to
\[2^{-2}(q_{2\lambda_1+4}(x,x)q_2(x,x)-2q_{2\lambda_1+5}(x,x)q_1(x,x)+2q_{2\lambda_1+6}(x,x))\]
and the identity (\ref{cSS=Q,even,n=2,l(lambda)=1}) is proved.
The identity (\ref{cSS=Q,odd,n=2,l(lambda)=1}) can also be proved in a similar way, and the theorem holds.
\end{proof}
\section{Consideration of the conjecture in general case}
In this section we consider Conjecture \ref{conj:cSS=Q} in general $n$.
Let $y=(y_1,y_2,\ldots,y_n)$ be a sequence of $n$ variables which is independent of $x=(x_1,x_2,\ldots,x_n)$, and write $y^2=(y_1^2,y_2^2,\ldots,y_n^2)$.
Since $\{s_\lambda(y^2)=s_\lambda(y_1^2,y_2^2,\ldots,y_n^2)\;|\;\lambda\in\mathscr P^{(n)}\}$ is linearly independent, the following conjecture is equivalent to Conjecture \ref{conj:cSS=Q}:
\begin{conj}\label{conj:cSS=Q'}
For $n\geq1$ the following identities hold:
\begin{align}\label{scSS=sQ,even}
\begin{split}
\sum_{\lambda\in\mathscr P^{(n)}}s_\lambda(y^2)\sum_{\mu,\nu\in\mathscr P^{(n)}}c^\lambda_{\mu\nu}S_{2\mu+\Delta}(x)S_{2\nu+\Delta}(x)
&=\sum_{\lambda\in\mathscr P^{(n)}}s_\lambda(y^2)(2^{-n}Q_{2\lambda+2\Delta}(x,x)),
\end{split}
\end{align}
\begin{align}\label{scSS=sQ,odd}
\begin{split}
\sum_{\lambda\in\mathscr P^{(n)}}s_\lambda(y^2)\sum_{\mu,\nu\in\mathscr P^{(n)}}c^\lambda_{\mu\nu}S_{2\mu+\delta}(x)S_{2\nu+\Delta}(x)
&=\sum_{\lambda\in\mathscr P^{(n)}}s_\lambda(y^2)(2^{-n}Q_{2\lambda+\Delta+\delta}(x,x)).
\end{split}
\end{align}
\end{conj}
The goal of the section is to write the both sides of the above identities as linear combinations of elementary symmetric polynomials $\{e_\xi(y^2)\mid\xi\in\mathscr P^{(n)}\}$, without Littlewood--Richardson coefficients.
Let $A_n$ be the set of sequences of $n$ non-negative integers $u=(u_1,u_2,\ldots,u_n)$ which consist of distinct terms $u_1,u_2,\ldots,u_n$.
Given $u\in A_n$, there is a unique permutation $\tau_u\in\mathfrak S_n$ such that $u_{\tau_u(n)}<\cdots<u_{\tau_u(2)}<u_{\tau_u(1)}$ hold.
By using these notations we can state the propositions which we prove in this section:
\begin{prop}\label{prop:LHSof(scSS)}
The left-hand sides of (\ref{scSS=sQ,even}) and (\ref{scSS=sQ,odd}) are written as follows:
\begin{align*}
&\sum_{\lambda\in\mathscr P^{(n)}}s_\lambda(y^2)\sum_{\mu,\nu\in\mathscr P^{(n)}}c^\lambda_{\mu\nu}S_{2\mu+\Delta}(x)S_{2\nu+\Delta}(x)\\
&=\sum_{\xi\in\widetilde{\mathscr P}^{(n)}}e_\xi(y^2)\sum_{u,v\in A_n}\mathrm{sgn}(\tau_u)\mathrm{sgn}(\tau_v)\left(\prod_{j=1}^nq_{2u_j+1-(n-j)}(x)q_{2v_j+1-(n-j)}(x)\right)\\
&\qquad\cdot\sum_{\substack{\eta,\zeta\in\widetilde{\mathscr P}^{(n)}\\\eta\cup\zeta=\xi}}K^{-1}_{\eta(\tau_u^{-1}u-\delta)'}K^{-1}_{\zeta(\tau_v^{-1}v-\delta)'},
\end{align*}
\begin{align*}
&\sum_{\lambda\in\mathscr P^{(n)}}s_\lambda(y^2)\sum_{\mu,\nu\in\mathscr P^{(n)}}c^\lambda_{\mu\nu}S_{2\mu+\delta}(x)S_{2\nu+\Delta}(x)\\
&=\sum_{\xi\in\widetilde{\mathscr P}^{(n)}}e_\xi(y^2)\sum_{u,v\in A_n}\mathrm{sgn}(\tau_u)\mathrm{sgn}(\tau_v)\left(\prod_{j=1}^nq_{2u_j-(n-j)}(x)q_{2v_j+1-(n-j)}(x)\right)\\
&\qquad\cdot\sum_{\substack{\eta,\zeta\in\widetilde{\mathscr P}^{(n)}\\\eta\cup\zeta=\xi}}K^{-1}_{\eta(\tau_u^{-1}u-\delta)'}K^{-1}_{\zeta(\tau_v^{-1}v-\delta)'}.
\end{align*}
\end{prop}
\begin{prop}\label{prop:RHSof(scSS)}
The right-hand sides of (\ref{scSS=sQ,even}) and (\ref{scSS=sQ,odd}) are written as follows:
\begin{align*}
\sum_{\lambda\in\mathscr P^{(n)}}s_\lambda(y^2)(2^{-n}Q_{2\lambda+2\Delta}(x,x))
&=\sum_{\xi\in\widetilde{\mathscr P}^{(n)}}e_\xi(y^2)\sum_{\lambda\in\mathscr P^{(n)}}2^{-n}K^{-1}_{\xi\lambda'}Q_{2\lambda+2\Delta}(x,x),\\
\sum_{\lambda\in\mathscr P^{(n)}}s_\lambda(y^2)(2^{-n}Q_{2\lambda+\Delta+\delta}(x,x))
&=\sum_{\xi\in\widetilde{\mathscr P}^{(n)}}e_\xi(y^2)\sum_{\lambda\in\mathscr P^{(n)}}2^{-n}K^{-1}_{\xi\lambda'}Q_{2\lambda+\Delta+\delta}(x,x).
\end{align*}
\end{prop}
By rewriting (\ref{scSS=sQ,even}) and (\ref{scSS=sQ,odd}) using Propositions \ref{prop:LHSof(scSS)} and \ref{prop:RHSof(scSS)}, we obtain that the following identities are equivalent to Conjecture \ref{conj:cSS=Q'}:
\begin{conj}\label{conj:cSS=Q''}
For $n\geq1$ and $\xi\in\widetilde{\mathscr P}^{(n)}$ the following identities hold:
\begin{align*}
&\sum_{u,v\in A_n}\mathrm{sgn}(\tau_u)\mathrm{sgn}(\tau_v)\left(\prod_{j=1}^nq_{2u_j+1-(n-j)}(x)q_{2v_j+1-(n-j)}(x)\right)
\sum_{\substack{\eta,\zeta\in\widetilde{\mathscr P}^{(n)}\\\eta\cup\zeta=\xi}}K^{-1}_{\eta(\tau_u^{-1}u-\delta)'}K^{-1}_{\zeta(\tau_v^{-1}v-\delta)'}\\
&=\sum_{\lambda\in\mathscr P^{(n)}}2^{-n}K^{-1}_{\xi\lambda'}Q_{2\lambda+2\Delta}(x,x),\\
&\sum_{u,v\in A_n}\mathrm{sgn}(\tau_u)\mathrm{sgn}(\tau_v)\left(\prod_{j=1}^nq_{2u_j-(n-j)}(x)q_{2v_j+1-(n-j)}(x)\right)
\sum_{\substack{\eta,\zeta\in\widetilde{\mathscr P}^{(n)}\\\eta\cup\zeta=\xi}}K^{-1}_{\eta(\tau_u^{-1}u-\delta)'}K^{-1}_{\zeta(\tau_v^{-1}v-\delta)'}\\
&=\sum_{\lambda\in\mathscr P^{(n)}}2^{-n}K^{-1}_{\xi\lambda'}Q_{2\lambda+\Delta+\delta}(x,x).
\end{align*}
\end{conj}
Proposition \ref{prop:RHSof(scSS)} comes immediately from Proposition \ref{prop:M101}.
Indeed, the former identity is proved as follows:
\begin{align*}
\sum_{\lambda\in\mathscr P^{(n)}}s_\lambda(y^2)(2^{-n}Q_{2\lambda+2\Delta}(x,x))
&=\sum_{\lambda\in\mathscr P^{(n)}}\left(\sum_{\xi\in\widetilde{\mathscr P}^{(n)}}K^{-1}_{\xi\lambda'}e_\xi(y^2)\right)(2^{-n}Q_{2\lambda+2\Delta}(x,x))\\
&=\sum_{\xi\in\widetilde{\mathscr P}^{(n)}}e_\xi(y^2)\sum_{\lambda\in\mathscr P^{(n)}}2^{-n}K^{-1}_{\xi\lambda'}Q_{2\lambda+2\Delta}(x,x).
\end{align*}
The latter one is obtained in a similar way.
\par The rest of this section is devoted to proving Proposition \ref{prop:LHSof(scSS)}.
Using the definition of the Littlewood--Richardson coefficients, we obtain
\begin{align*}
&\sum_{\lambda\in\mathscr P^{(n)}}s_\lambda(y^2)\sum_{\mu,\nu\in\mathscr P^{(n)}}c^\lambda_{\mu\nu}S_{2\mu+\Delta}(x)S_{2\nu+\Delta}(x)\\
&=\sum_{\mu,\nu\in\mathscr P^{(n)}}\left(\sum_{\lambda\in\mathscr P^{(n)}}c^\lambda_{\mu\nu}s_\lambda(y^2)\right)S_{2\mu+\Delta}(x)S_{2\nu+\Delta}(x)\\
&=\sum_{\mu,\nu\in\mathscr P^{(n)}}s_\mu(y^2)s_\nu(y^2)S_{2\mu+\Delta}(x)S_{2\nu+\Delta}(x)\\
&=\left(\sum_{\mu\in\mathscr P^{(n)}}s_\mu(y^2)S_{2\mu+\Delta}(x)\right)\left(\sum_{\nu\in\mathscr P^{(n)}}s_\nu(y^2)S_{2\nu+\Delta}(x)\right).
\end{align*}
Now we transform
\begin{equation}\label{sS}
\sum_{\mu\in\mathscr P^{(n)}}s_\mu(y^2)S_{2\mu+\Delta}(x)
\end{equation}
to a linear combination of $\{e_\xi(y^2)\mid\xi\in\widetilde{\mathscr P}^{(n)}\}$.
Since
\[y_1y_2\cdots y_na_\delta(y^2)s_\mu(y^2)=y_1y_2\cdots y_na_{\mu+\delta}(y^2)=y_1y_2\cdots y_na_{2\mu+2\delta}(y)=a_{2\mu+\Delta+\delta}(y)\]
for $\mu\in\mathscr P^{(n)}$, we obtain
\begin{align*}
&y_1y_2\cdots y_na_\delta(y^2)\sum_{\mu\in\mathscr P^{(n)}}s_\mu(y^2)S_{2\mu+\Delta}(x)
=\sum_{\mu\in\mathscr P^{(n)}}a_{2\mu+\Delta+\delta}(y)S_{2\mu+\Delta}(x)\notag\\
&=\mathrm{TO}_y\left(\sum_{\rho\in\mathscr P^{(n)}}a_{\rho+\delta}(y)S_\rho(x)\right).
\end{align*}
The summation in the parenthesis can be calculated as follows:
\begin{align*}
&\sum_{\rho\in\mathscr P^{(n)}}a_{\rho+\delta}(y)S_\rho(x)
=a_\delta(y)\sum_{\rho\in\mathscr P^{(n)}}s_\rho(y)S_\rho(x)
=a_\delta(y)\prod_{i=1}^n\prod_{j=1}^n\dfrac{1+x_iy_j}{1-x_iy_j}\\
&=a_\delta(y)\prod_{j=1}^n\sum_{r_j=0}^\infty q_{r_j}(x)y_j^{r_j}
=a_\delta(y)\sum_{r_1,r_2,\ldots,r_n\geq0}\prod_{j=1}^nq_{r_j}(x)y_j^{r_j}\\
&=\left(\sum_{\sigma\in\mathfrak S_n}\mathrm{sgn}(\sigma)\sigma(y^\delta)\right)\sum_{r_1,r_2,\ldots,r_n\geq0}\prod_{j=1}^nq_{r_j}(x)y_j^{r_j}\\
&=\sum_{\sigma\in\mathfrak S_n}\mathrm{sgn}(\sigma)y_{\sigma(1)}^{n-1}y_{\sigma(2)}^{n-2}\cdots y_{\sigma(n)}^{n-n}\sum_{r_1,r_2,\ldots,r_n\geq0}\prod_{j=1}^nq_{r_{\sigma(j)}}(x)y_{\sigma(j)}^{r_{\sigma(j)}}\\
&=\sum_{\sigma\in\mathfrak S_n}\mathrm{sgn}(\sigma)\sum_{r_1,r_2,\ldots,r_n\geq0}\prod_{j=1}^nq_{r_{\sigma(j)}}(x)y_{\sigma(j)}^{n-j+r_{\sigma(j)}}\\
&=\sum_{\sigma\in\mathfrak S_n}\mathrm{sgn}(\sigma)\sum_{r_1,r_2,\ldots,r_n\geq0}\prod_{j=1}^nq_{(n-j+r_{\sigma(j)})-(n-j)}(x)y_{\sigma(j)}^{n-j+r_{\sigma(j)}}\\
&=\sum_{\sigma\in\mathfrak S_n}\mathrm{sgn}(\sigma)\sum_{s_1=n-\sigma^{-1}(1)}^\infty\sum_{s_2=n-\sigma^{-1}(2)}^\infty\cdots\sum_{s_n=n-\sigma^{-1}(n)}^\infty\prod_{j=1}^nq_{s_{\sigma(j)}-(n-j)}(x)y_{\sigma(j)}^{s_{\sigma(j)}}\\
&=\sum_{\sigma\in\mathfrak S_n}\mathrm{sgn}(\sigma)\sum_{s_1,s_2,\ldots,s_n\geq0}\prod_{j=1}^nq_{s_{\sigma(j)}-(n-j)}(x)y_{\sigma(j)}^{s_{\sigma(j)}},
\end{align*}
using Proposition \ref{prop:M225(4.7)} for the second equality.
The last equality is valid because $q_r(x)=0$ for $r<0$.
We write $s\coloneqq(s_1,s_2,\ldots,s_n)$ and obtain
\[\sum_{\rho\in\mathscr P^{(n)}}a_{\rho+\delta}(y)S_\rho(x)=\sum_{\sigma\in\mathfrak S_n}\mathrm{sgn}(\sigma)\sum_{s\in(\mathbb Z_{\geq0})^n}\prod_{j=1}^nq_{s_{\sigma(j)}-(n-j)}(x)y_{\sigma(j)}^{s_{\sigma(j)}}.\]
Therefore, we obtain
\begin{align*}
&y_1y_2\cdots y_na_\delta(y^2)\sum_{\mu\in\mathscr P^{(n)}}s_\mu(y^2)S_{2\mu+\Delta}(x)\\
&=\mathrm{TO}_y\left(\sum_{\rho\in\mathscr P^{(n)}}a_{\rho+\delta}(y)S_\rho(x)\right)\\
&=\mathrm{TO}_y\left(\sum_{\sigma\in\mathfrak S_n}\mathrm{sgn}(\sigma)\sum_{s\in(\mathbb Z_{\geq0})^n}\prod_{j=1}^nq_{s_{\sigma(j)}-(n-j)}(x)y_{\sigma(j)}^{s_{\sigma(j)}}\right)\\
&=\sum_{\sigma\in\mathfrak S_n}\mathrm{sgn}(\sigma)\sum_{u\in(\mathbb Z_{\geq0})^n}\prod_{j=1}^nq_{2u_{\sigma(j)}+1-(n-j)}(x)y_{\sigma(j)}^{2u_{\sigma(j)}+1},
\end{align*}
and we have
\begin{align*}
&a_\delta(y^2)\sum_{\mu\in\mathscr P^{(n)}}s_\mu(y^2)S_{2\mu+\Delta}(x)\\
&=\sum_{\sigma\in\mathfrak S_n}\mathrm{sgn}(\sigma)\sum_{u\in(\mathbb Z_{\geq0})^n}\prod_{j=1}^nq_{2u_{\sigma(j)}+1-(n-j)}(x)y_{\sigma(j)}^{2u_{\sigma(j)}}\\
&=\sum_{\sigma\in\mathfrak S_n}\mathrm{sgn}(\sigma)\sum_{u\in(\mathbb Z_{\geq0})^n}\prod_{j=1}^nq_{2u_j+1-(n-j)}(x)y_{\sigma(j)}^{2u_j}\\
&=\sum_{u\in(\mathbb Z_{\geq0})^n}\left(\prod_{j=1}^nq_{2u_j+1-(n-j)}(x)\right)\sum_{\sigma\in\mathfrak S_n}\mathrm{sgn}(\sigma)\left(\prod_{j=1}^ny_{\sigma(j)}^{2u_j}\right)\\
&=\sum_{u\in(\mathbb Z_{\geq0})^n}\left(\prod_{j=1}^nq_{2u_j+1-(n-j)}(x)\right)\sum_{\sigma\in\mathfrak S_n}\mathrm{sgn}(\sigma)\sigma(y^{2u})\\
&=\sum_{u\in(\mathbb Z_{\geq0})^n}\left(\prod_{j=1}^nq_{2u_j+1-(n-j)}(x)\right)a_u(y^2).
\end{align*}
Here it suffices to sum up for $u\in A_n$ because $a_u$ is antisymmetric.
Using $\tau_u$ one obtains
\[a_u(y^2)=\mathrm{sgn}(\tau_u)a_{\tau_u^{-1}u}(y^2),\]
where
\[\tau_u^{-1}u=\tau_u^{-1}(u_1,u_2,\ldots,u_n)\coloneqq(u_{\tau_u(1)},u_{\tau_u(2)},\ldots,u_{\tau_u(n)}),\]
and
\[\dfrac{a_u(y^2)}{a_\delta(y^2)}=\mathrm{sgn}(\tau_u)s_{\tau_u^{-1}u-\delta}(y^2)\]
holds.
Therefore,
\begin{align*}
&\sum_{\mu\in\mathscr P^{(n)}}s_\mu(y^2)S_{2\mu+\Delta}(x)\\
&=\sum_{u\in A_n}\left(\prod_{j=1}^nq_{2u_j+1-(n-j)}(x)\right)\dfrac{a_u(y^2)}{a_\delta(y^2)}\\
&=\sum_{u\in A_n}\left(\prod_{j=1}^nq_{2u_j+1-(n-j)}(x)\right)\mathrm{sgn}(\tau_u)s_{\tau_u^{-1}u-\delta}(y^2)\\
&=\sum_{u\in A_n}\left(\prod_{j=1}^nq_{2u_j+1-(n-j)}(x)\right)\mathrm{sgn}(\tau_u)\sum_{\xi\in\widetilde{\mathscr P}^{(n)}}K^{-1}_{\xi(\tau_u^{-1}u-\delta)'}e_\xi(y^2),
\end{align*}
using Proposition \ref{prop:M101} for the last equality.
Now the summation (\ref{sS}) is written as a linear combination of elementary symmetric polynomials, and we can write the left-hand side of (\ref{scSS=sQ,even}) in such form:
\begin{align*}
&\sum_{\lambda\in\mathscr P^{(n)}}s_\lambda(y^2)\sum_{\mu,\nu\in\mathscr P^{(n)}}c^\lambda_{\mu\nu}S_{2\mu+\Delta}(x)S_{2\nu+\Delta}(x)\\
&=\left(\sum_{\mu\in\mathscr P^{(n)}}s_\mu(y^2)S_{2\mu+\Delta}(x)\right)\left(\sum_{\nu\in\mathscr P^{(n)}}s_\nu(y^2)S_{2\nu+\Delta}(x)\right)\\
&=\left(\sum_{u\in A_n}\left(\prod_{j=1}^nq_{2u_j+1-(n-j)}(x)\right)\mathrm{sgn}(\tau_u)\sum_{\eta\in\widetilde{\mathscr P}^{(n)}}K^{-1}_{\eta(\tau_u^{-1}u-\delta)'}e_\eta(y^2)\right)\\
&\quad\cdot\left(\sum_{v\in A_n}\left(\prod_{j=1}^nq_{2v_j+1-(n-j)}(x)\right)\mathrm{sgn}(\tau_v)\sum_{\zeta\in\widetilde{\mathscr P}^{(n)}}K^{-1}_{\zeta(\tau_v^{-1}v-\delta)'}e_\zeta(y^2)\right)\\
&=\sum_{u,v\in A_n}\mathrm{sgn}(\tau_u)\mathrm{sgn}(\tau_v)\left(\prod_{j=1}^nq_{2u_j+1-(n-j)}(x)q_{2v_j+1-(n-j)}(x)\right)\\
&\qquad\cdot\sum_{\eta,\zeta\in\widetilde{\mathscr P}^{(n)}}K^{-1}_{\eta(\tau_u^{-1}u-\delta)'}K^{-1}_{\zeta(\tau_v^{-1}v-\delta)'}e_\eta(y^2)e_\zeta(y^2)\\
&=\sum_{u,v\in A_n}\mathrm{sgn}(\tau_u)\mathrm{sgn}(\tau_v)\left(\prod_{j=1}^nq_{2u_j+1-(n-j)}(x)q_{2v_j+1-(n-j)}(x)\right)\\
&\qquad\cdot\sum_{\eta,\zeta\in\widetilde{\mathscr P}^{(n)}}K^{-1}_{\eta(\tau_u^{-1}u-\delta)'}K^{-1}_{\zeta(\tau_v^{-1}v-\delta)'}e_{\eta\cup\zeta}(y^2)\\
&=\sum_{u,v\in A_n}\mathrm{sgn}(\tau_u)\mathrm{sgn}(\tau_v)\left(\prod_{j=1}^nq_{2u_j+1-(n-j)}(x)q_{2v_j+1-(n-j)}(x)\right)\\
&\qquad\cdot\sum_{\xi\in\widetilde{\mathscr P}^{(n)}}\sum_{\substack{\eta,\zeta\in\widetilde{\mathscr P}^{(n)}\\\eta\cup\zeta=\xi}}K^{-1}_{\eta(\tau_u^{-1}u-\delta)'}K^{-1}_{\zeta(\tau_v^{-1}v-\delta)'}e_\xi(y^2)\\
&=\sum_{\xi\in\widetilde{\mathscr P}^{(n)}}e_\xi(y^2)\sum_{u,v\in A_n}\mathrm{sgn}(\tau_u)\mathrm{sgn}(\tau_v)\left(\prod_{j=1}^nq_{2u_j+1-(n-j)}(x)q_{2v_j+1-(n-j)}(x)\right)\\
&\qquad\cdot\sum_{\substack{\eta,\zeta\in\widetilde{\mathscr P}^{(n)}\\\eta\cup\zeta=\xi}}K^{-1}_{\eta(\tau_u^{-1}u-\delta)'}K^{-1}_{\zeta(\tau_v^{-1}v-\delta)'}.
\end{align*}
It completes the proof of the former equality of Proposition \ref{prop:LHSof(scSS)}, and the latter one can be proved similarly.
\section{A proof of the conjecture in general case when $n=2$}
We prove Conjecture \ref{conj:cSS=Q} in general case when $n=2$ in the section using the results in the previous section.
The following theorem is the case $n=2$ of Conjecture \ref{conj:cSS=Q}, and is the main theorem of the present paper:
\begin{thm}\label{thm:cSS=Q,n=2}
Conjecture \ref{conj:cSS=Q} holds for $n=2$;
\begin{align*}
\sum_{\mu,\nu\in\mathscr P^{(2)}}c^\lambda_{\mu\nu}S_{2\mu+\Delta}(x)S_{2\nu+\Delta}(x)&=2^{-2}Q_{2\lambda+2\Delta}(x,x),\\
\sum_{\mu,\nu\in\mathscr P^{(2)}}c^\lambda_{\mu\nu}S_{2\mu+\delta}(x)S_{2\nu+\Delta}(x)&=2^{-2}Q_{2\lambda+\Delta+\delta}(x,x)
\end{align*}
hold for $\lambda\in\mathscr P^{(2)}$, where $\Delta\coloneqq(2,1)$ and $\delta\coloneqq(1,0)$.
\end{thm}
To prove the theorem, it suffices to prove the following theorem which is the case $n=2$ of Conjecture \ref{conj:cSS=Q''}:
\begin{thm}\label{thm:cSS=Q'',n=2}
For $\xi\in\widetilde{\mathscr P}^{(2)}$ the following identities hold:
\begin{align*}
&\sum_{u,v\in A_2}\mathrm{sgn}(\tau_u)\mathrm{sgn}(\tau_v)q_{2u_1}(x)q_{2u_2+1}(x)q_{2v_1}(x)q_{2v_2+1}(x)
\sum_{\substack{\eta,\zeta\in\widetilde{\mathscr P}^{(2)}\\\eta\cup\zeta=\xi}}K^{-1}_{\eta(\tau_u^{-1}u-\delta)'}K^{-1}_{\zeta(\tau_v^{-1}v-\delta)'}\\
&=\sum_{\lambda\in\mathscr P^{(2)}}2^{-2}K^{-1}_{\xi\lambda'}Q_{2\lambda+2\Delta}(x,x),\\
&\sum_{u,v\in A_2}\mathrm{sgn}(\tau_u)\mathrm{sgn}(\tau_v)q_{2u_1-1}(x)q_{2u_2}(x)q_{2v_1}(x)q_{2v_2+1}(x)
\sum_{\substack{\eta,\zeta\in\widetilde{\mathscr P}^{(2)}\\\eta\cup\zeta=\xi}}K^{-1}_{\eta(\tau_u^{-1}u-\delta)'}K^{-1}_{\zeta(\tau_v^{-1}v-\delta)'}\\
&=\sum_{\lambda\in\mathscr P^{(2)}}2^{-2}K^{-1}_{\xi\lambda'}Q_{2\lambda+\Delta+\delta}(x,x),
\end{align*}
where $A_2$ is the set of $u=(u_1,u_2)\in(\mathbb Z_{\geq0})^2$ such that $u_1\ne u_2$, and $\tau_u\in\mathfrak S_2$ satisfies $u_{\tau_u(2)}<u_{\tau_u(1)}$ for $u\in A_2$.
\end{thm}
We prove the theorems by giving a new expression of Schur's $Q$-functions corresponding to partitions of length $2$, and considering combinatorics of the inverse Kostka matrix.
First, we give a new expression of Schur's $Q$-functions.
Given $\lambda=(\lambda_1,\lambda_2)\in\mathscr P^{(2)}$ and $\sigma_1,\sigma_2\in\mathfrak S_2$, let $B(\lambda;\sigma_1,\sigma_2)$ be the set of pairs of pairs of non-negative integers $(u,v)=((u_1,u_2),(v_1,v_2))\in(\mathbb Z_{\geq0})^2\times(\mathbb Z_{\geq0})^2$ which satisfy $u_1+u_2+v_1+v_2=\lambda_1+\lambda_2+2$ and $u_{\sigma_1(2)}+v_{\sigma_2(2)}\leq\lambda_2$.
\begin{prop}\label{prop:expression_of_Q}
For $\lambda\in\mathscr P^{(2)}$, the following equalities hold:
\begin{align*}
Q_{2\lambda+2\Delta}(x,x)
&=4\sum_{\sigma_1,\sigma_2\in\mathfrak S_2}\mathrm{sgn}(\sigma_1)\mathrm{sgn}(\sigma_2)\sum_{(u,v)\in B(\lambda;\sigma_1,\sigma_2)}q_{2u_1}(x)q_{2u_2+1}(x)q_{2v_1}(x)q_{2v_2+1}(x),\\
Q_{2\lambda+\Delta+\delta}(x,x)
&=4\sum_{\sigma_1,\sigma_2\in\mathfrak S_2}\mathrm{sgn}(\sigma_1)\mathrm{sgn}(\sigma_2)\sum_{(u,v)\in B(\lambda;\sigma_1,\sigma_2)}q_{2u_1}(x)q_{2u_2+1}(x)q_{2v_1-1}(x)q_{2v_2}(x).
\end{align*}
\end{prop}
\begin{proof}
From the definition of Schur's $Q$-functions we obtain the following calculation:
\begin{align*}
Q_{2\lambda+2\Delta}(x,x)
&=Q_{(2\lambda_1+4,2\lambda_2+2)}(x,x)\\
&=q_{2\lambda_1+4}(x,x)q_{2\lambda_2+2}(x,x)+2\sum_{i=1}^{2\lambda_2+2}(-1)^iq_{2\lambda_1+4+i}(x,x)q_{2\lambda_2+2-i}(x,x)\\
&=q_{2\lambda_1+4}(x,x)q_{2\lambda_2+2}(x,x)-2\sum_{j=0}^{\lambda_2}q_{2\lambda_1+5+2j}(x,x)q_{2\lambda_2+1-2j}(x,x)\\
&\quad+2\sum_{j=1}^{\lambda_2+1}q_{2\lambda_1+4+2j}(x,x)q_{2\lambda_2+2-2j}(x,x).
\end{align*}
Using (\ref{q(x,x)}) and Lemma \ref{lem:q(x,x)=2sumq(x)q(x)},
\begin{align*}
&Q_{2\lambda+2\Delta}(x,x)\\
&=2q_{2\lambda_1+4}(x,x)\sum_{k=0}^{\lambda_2}q_{2k+1}(x)q_{2\lambda_2-2k+1}(x)\\
&\quad-2\sum_{j=0}^{\lambda_2}q_{2\lambda_1+5+2j}(x,x)\sum_{i=0}^{2\lambda_2+1-2j}q_i(x)q_{2\lambda_2+1-2j-i}(x)\\
&\quad+2\sum_{j=1}^{\lambda_2+1}q_{2\lambda_1+4+2j}(x,x)\sum_{i=0}^{2\lambda_2+2-2j}q_i(x)q_{2\lambda_2+2-2j-i}(x)
\end{align*}
holds, and
\begin{align*}
&Q_{2\lambda+2\Delta}(x,x)\\
&=2q_{2\lambda_1+4}(x,x)\sum_{k=0}^{\lambda_2}q_{2k+1}(x)q_{2\lambda_2-2k+1}(x)\\
&\quad-2\sum_{j=0}^{\lambda_2}q_{2\lambda_1+5+2j}(x,x)\sum_{u_1=0}^{\lambda_2-j}q_{2u_1}(x)q_{2\lambda_2+1-2j-2u_1}(x)\\
&\quad-2\sum_{j=0}^{\lambda_2}q_{2\lambda_1+5+2j}(x,x)\sum_{u_2=0}^{\lambda_2-j}q_{2u_2+1}(x)q_{2\lambda_2-2j-2u_2}(x)\\
&\quad+2\sum_{j=1}^{\lambda_2+1}q_{2\lambda_1+4+2j}(x,x)\sum_{u_1=0}^{\lambda_2+1-j}q_{2u_1}(x)q_{2\lambda_2+2-2j-2u_1}(x)\\
&\quad+2\sum_{j=1}^{\lambda_2+1}q_{2\lambda_1+4+2j}(x,x)\sum_{u_2=0}^{\lambda_2-j}q_{2u_2+1}(x)q_{2\lambda_2+1-2j-2u_2}(x)\\
&=-2\sum_{j=0}^{\lambda_2}q_{2\lambda_1+5+2j}(x,x)\sum_{u_1=0}^{\lambda_2-j}q_{2u_1}(x)q_{2\lambda_2+1-2j-2u_1}(x)\\
&\quad-2\sum_{j=0}^{\lambda_2}q_{2\lambda_1+5+2j}(x,x)\sum_{u_2=0}^{\lambda_2-j}q_{2u_2+1}(x)q_{2\lambda_2-2j-2u_2}(x)\\
&\quad+2\sum_{j=1}^{\lambda_2+1}q_{2\lambda_1+4+2j}(x,x)\sum_{u_1=0}^{\lambda_2+1-j}q_{2u_1}(x)q_{2\lambda_2+2-2j-2u_1}(x)\\
&\quad+2\sum_{j=0}^{\lambda_2+1}q_{2\lambda_1+4+2j}(x,x)\sum_{u_2=0}^{\lambda_2-j}q_{2u_2+1}(x)q_{2\lambda_2+1-2j-2u_2}(x).
\end{align*}
The first summation can be calculated as follows:
\begin{align*}
&-2\sum_{j=0}^{\lambda_2}q_{2\lambda_1+5+2j}(x,x)\sum_{u_1=0}^{\lambda_2-j}q_{2u_1}(x)q_{2\lambda_2+1-2j-2u_1}(x)\\
&=-2\sum_{u_1=0}^{\lambda_2}\sum_{j=0}^{\lambda_2-u_1}q_{2\lambda_1+5+2j}(x,x)q_{2u_1}(x)q_{2\lambda_2+1-2j-2u_1}(x)\\
&=-2\sum_{u_1=0}^{\lambda_2}\sum_{v_2=0}^{\lambda_2-u_1}q_{2\lambda_1+5+2(\lambda_2-u_1-v_2)}(x,x)q_{2u_1}(x)q_{2\lambda_2+1-2(\lambda_2-u_1-v_2)-2u_1}(x)\\
&=-2\sum_{u_1=0}^{\lambda_2}\sum_{v_2=0}^{\lambda_2-u_1}q_{2(\lambda_1+\lambda_2-u_1-v_2+2)+1}(x,x)q_{2u_1}(x)q_{2v_2+1}(x).
\end{align*}
Using Lemma \ref{lem:q(x,x)=2sumq(x)q(x)} again, one obtains
\begin{align*}
&-2\sum_{j=0}^{\lambda_2}q_{2\lambda_1+5+2j}(x,x)\sum_{u_1=0}^{\lambda_2-j}q_{2u_1}(x)q_{2\lambda_2+1-2j-2u_1}(x)\\
&=-4\sum_{u_1=0}^{\lambda_2}\sum_{v_2=0}^{\lambda_2-u_1}\sum_{u_2=0}^{\lambda_1+\lambda_2-u_1-v_2+2}q_{2u_2+1}(x)q_{2(\lambda_1+\lambda_2-u_1-u_2-v_2+2)}(x)q_{2u_1}(x)q_{2v_2+1}(x)\\
&=-4\sum_{(u,v)\in B(\lambda;(1,2),\mathrm{id})}q_{2u_1}(x)q_{2u_2+1}(x)q_{2v_1}(x)q_{2v_2+1}(x),
\end{align*}
where we write the elements of $\mathfrak S_2$ as $\mathfrak S_2=\{\mathrm{id},(1,2)\}$.
The other summations can also be calculated similarly as
\begin{align*}
&-2\sum_{j=0}^{\lambda_2}q_{2\lambda_1+5+2j}(x,x)\sum_{u_2=0}^{\lambda_2-j}q_{2u_2+1}(x)q_{2\lambda_2-2j-2u_2}(x)\\
&=-4\sum_{(u,v)\in B(\lambda;\mathrm{id},(1,2))}q_{2u_1}(x)q_{2u_2+1}(x)q_{2v_1}(x)q_{2v_2+1}(x),
\end{align*}
\begin{align*}
&2\sum_{j=1}^{\lambda_2+1}q_{2\lambda_1+4+2j}(x,x)\sum_{u_1=0}^{\lambda_2+1-j}q_{2u_1}(x)q_{2\lambda_2+2-2j-2u_1}(x)\\
&=4\sum_{(u,v)\in B(\lambda;(1,2),(1,2))}q_{2u_1}(x)q_{2u_2+1}(x)q_{2v_1}(x)q_{2v_2+1}(x),
\end{align*}
and
\begin{align*}
&2\sum_{j=0}^{\lambda_2+1}q_{2\lambda_1+4+2j}(x,x)\sum_{u_2=0}^{\lambda_2-j}q_{2u_2+1}(x)q_{2\lambda_2+1-2j-2u_2}(x)\\
&=4\sum_{(u,v)\in B(\lambda;\mathrm{id},\mathrm{id})}q_{2u_1}(x)q_{2u_2+1}(x)q_{2v_1}(x)q_{2v_2+1}(x).
\end{align*}
Therefore, we obtain
\begin{align*}
&Q_{2\lambda+2\Delta}(x,x)\\
&=-4\sum_{(u,v)\in B(\lambda;(1,2),\mathrm{id})}q_{2u_1}(x)q_{2u_2+1}(x)q_{2v_1}(x)q_{2v_2+1}(x)\\
&-4\sum_{(u,v)\in B(\lambda;\mathrm{id},(1,2))}q_{2u_1}(x)q_{2u_2+1}(x)q_{2v_1}(x)q_{2v_2+1}(x)\\
&+4\sum_{(u,v)\in B(\lambda;(1,2),(1,2))}q_{2u_1}(x)q_{2u_2+1}(x)q_{2v_1}(x)q_{2v_2+1}(x)\\
&+4\sum_{(u,v)\in B(\lambda;\mathrm{id},\mathrm{id})}q_{2u_1}(x)q_{2u_2+1}(x)q_{2v_1}(x)q_{2v_2+1}(x)\\
&=4\sum_{\sigma_1,\sigma_2\in\mathfrak S_2}\mathrm{sgn}(\sigma_1)\mathrm{sgn}(\sigma_2)\sum_{(u,v)\in B(\lambda;\sigma_1,\sigma_2)}q_{2u_1}(x)q_{2u_2+1}(x)q_{2v_1}(x)q_{2v_2+1}(x)
\end{align*}
and it completes the proof of the former equality.
The latter equality can be proved in a similar way.
\end{proof}
Next, we consider the inverse Kostka matrix.
We start with a lemma on binomial coefficients.
Write $\binom ab\coloneqq0$ if $a<0$ or $b<0$.
\begin{lem}\label{lem:binom_coeff}
The following holds for non-negative integers $n,m$ and $0\leq k\leq n+m$:
\[\binom nm=\sum_{i=0}^m\binom{n-k+i}{m-i}\binom{k-i}{i}+\sum_{i=0}^{m-1}\binom{n-k+i}{m-i-1}\binom{k-i-1}{i}.\]
\end{lem}
\begin{proof}
We use induction on $k$.
When $k=0$ the statement is obvious.
Suppose that $k\geq1$ and that the statement holds for $k-1$.
From the induction hypothesis
\begin{align*}
\binom nm
&=\sum_{i=0}^m\binom{n-k+1+i}{m-i}\binom{k-1-i}{i}+\sum_{i=0}^{m-1}\binom{n-k+1+i}{m-i-1}\binom{k-i-2}{i}\\
&=\sum_{i=0}^m\binom{n-k+1+i}{m-i}\binom{k-1-i}{i}+\sum_{i=1}^m\binom{n-k+i}{m-i}\binom{k-1-i}{i-1}
\end{align*}
holds.
If $i=k-n-1$, then
\[\binom{n-k+1+i}{m-i}=\binom{0}{m+n-k+1}=0,\]
\[\binom{n-k+i}{m-i}+\binom{n-k+i}{m-i-1}=\binom{-1}{m+n-k+1}+\binom{-1}{m+n-k}=0\]
hold and we have
\[\binom{n-k+1+i}{m-i}=\binom{n-k+i}{m-i}+\binom{n-k+i}{m-i-1}.\]
Even if $i\ne k-n-1$, the equality
\[\binom{n-k+1+i}{m-i}=\binom{n-k+i}{m-i}+\binom{n-k+i}{m-i-1}\]
holds.
Therefore,
\begin{align*}
\binom nm&=\sum_{i=0}^m\binom{n-k+i}{m-i}\binom{k-1-i}{i}+\sum_{i=0}^{m-1}\binom{n-k+i}{m-i-1}\binom{k-1-i}{i}\\
&\qquad+\sum_{i=1}^m\binom{n-k+i}{m-i}\binom{k-1-i}{i-1}\\
&=\sum_{i=0}^m\binom{n-k+i}{m-i}\binom{k-i}{i}+\sum_{i=0}^{m-1}\binom{n-k+i}{m-i-1}\binom{k-1-i}{i}
\end{align*}
and the statement holds for $k$.
\end{proof}
We can interpret the lemma as the enumeration of partitions of continuous $n+m$ cells into $n-m$ single cells and $m$ pairs of adjacent two cells.
The number of such partitions is equal to $\binom nm$, which is the number of sequences of $n-m$ single cells and $m$ pairs of two cells.
Now we count the number of the partitions of cells which do not have the pair of the $k$th and the $(k+1)$st cells as its part.
The number of partitions of the first $k$ cells into $k-2i$ single cells and $i$ pairs of adjacent two cells is $\binom{k-i}{i}$, and the number of partitions of the last $n+m-k$ cells into $n-m-k+2i$ single cells and $m-i$ pairs of adjacent two cells is $\binom{n-k+i}{m-i}$.
Hence there are
\[\sum_{i=0}^m\binom{n-k+i}{m-i}\binom{k-i}{i}\]
such partitions of $n+m$ cells which do not have the pair of the $k$th and the $(k+1)$st cells as its part.
The number of such partitions which have the pair of the $k$th and the $(k+1)$st cells as its part can be enumerated similarly.
It is equal to the number of partitions of $n+m-2$ cells into $n-m$ single cells and $m-1$ pairs of adjacent two cells which do not have the pair of the $(k-1)$st and the $k$th cells as its part, and is equal to
\[\sum_{i=0}^{m-1}\binom{n-k+i}{m-i-1}\binom{k-i-1}{i}.\]
\par E\u gecio\u glu and Remmel~\cite{ER1990} interpreted the inverse Kostka matrix combinatorially and gave an exact formula for $K^{-1}_{\lambda\mu}$ for $\lambda,\mu\in\mathscr P^{(2)}$ as follows:
\begin{prop}[{\cite[Corollary 3]{ER1990}}]\label{prop:inv_Kostka_l=2}
For partitions $\lambda,\mu\in\widetilde{\mathscr P}^{(2)}$ such that $m_1(\lambda)+2m_2(\lambda)=m_1(\mu)+2m_2(\mu)$, it holds that
\[K^{-1}_{\lambda\mu}=(-1)^{m_2(\lambda)-m_2(\mu)}\binom{l(\lambda)-m_2(\mu)}{m_1(\lambda)}.\]
\end{prop}
For $\sigma_1,\sigma_2\in\mathfrak S_2$ and $u=(u_1,u_2), v=(v_1,v_2)\in(\mathbb Z_{\geq0})^2$, let $C(u,v;\sigma_1,\sigma_2)$ be the set of partitions $\lambda=(\lambda_1,\lambda_2)\in\mathscr P^{(2)}$ which satisfy the conditions $\lambda_1+\lambda_2=u_1+u_2+v_1+v_2-2$ and $\lambda_2\geq u_{\sigma_1(2)}+v_{\sigma_2(2)}$.
One obtains $C(\tau_1^{-1}u,\tau_2^{-1}v;\sigma_1,\sigma_2)=C(u,v;\tau_1\sigma_1,\tau_2\sigma_2)$ for $\tau_1,\tau_2\in\mathfrak S_2$ from the definition.
\begin{prop}\label{prop:inverse_Kostka}
For $\xi\in\widetilde{\mathscr P}^{(2)}$ and $u=(u_1,u_2),v=(v_1,v_2)\in(\mathbb Z_{\geq0})^2$ which satisfy $u_2<u_1$ and $v_2<v_1$, the following identity holds:
\[\sum_{\substack{\eta,\zeta\in\widetilde{\mathscr P}^{(2)}\\\eta\cup\zeta=\xi}}K^{-1}_{\eta(u_1-1,u_2)'}K^{-1}_{\zeta(v_1-1,v_2)'}=\sum_{\sigma_1,\sigma_2\in\mathfrak S_2}\mathrm{sgn}(\sigma_1)\mathrm{sgn}(\sigma_2)\sum_{\lambda\in C(u,v;\sigma_1,\sigma_2)}K^{-1}_{\xi\lambda'}.\]
\end{prop}
\begin{proof}
Since the statement is symmetric in $u$ and $v$, we assume $u_2+v_1\leq u_1+v_2$ without loss of generality.
We may assume $m_1(\xi)+2m_2(\xi)=u_1+u_2+v_1+v_2-2$ because the both sides equal $0$ otherwise.
In the left-hand side of the statement, non-zero summands are ones which correspond to $\eta$ and $\zeta$ such that $m_1(\eta)+2m_2(\eta)=u_1+u_2-1$ and $m_1(\zeta)+2m_2(\zeta)=v_1+v_2-1$. 
By the substitution $a=m_2(\zeta)$, one obtains $m_2(\eta)=m_2(\xi)-a$ and can rewrite the left-hand side of the statement as follows:
\[\sum_{\substack{\eta,\zeta\in\widetilde{\mathscr P}^{(2)}\\\eta\cup\zeta=\xi}}K^{-1}_{\eta(u_1-1,u_2)'}K^{-1}_{\zeta(v_1-1,v_2)'}=\sum_{a=0}^{m_2(\xi)}K^{-1}_{\eta_a(u_1-1,u_2)'}K^{-1}_{\zeta_a(v_1-1,v_2)'},\]
where $\eta_a=(1^{u_1+u_2-1-2(m_2(\xi)-a)}2^{m_2(\xi)-a})$ and $\zeta_a=(1^{v_1+v_2-1-2a}2^a)$ for $0\leq a\leq m_2(\xi)$.
It follows from Proposition \ref{prop:inv_Kostka_l=2} that
\[\sum_{\substack{\eta,\zeta\in\widetilde{\mathscr P}^{(2)}\\\eta\cup\zeta=\xi}}K^{-1}_{\eta(u_1-1,u_2)'}K^{-1}_{\zeta(v_1-1,v_2)'}=(-1)^{m_2(\xi)-u_2-v_2}\sum_{a=0}^{m_2(\xi)}\binom{u_1-1-m_2(\xi)+a}{m_2(\xi)-a-u_2}\binom{v_1-1-a}{a-v_2}.\]
The non-zero summands are ones which correspond to $v_2\leq a\leq m_2(\xi)-u_2$.
Therefore, by the substitution $i=a-v_2$, one obtains
\begin{align}\label{inv.kostka.formula1}
&\sum_{\substack{\eta,\zeta\in\widetilde{\mathscr P}^{(2)}\\\eta\cup\zeta=\xi}}K^{-1}_{\eta(u_1-1,u_2)'}K^{-1}_{\zeta(v_1-1,v_2)'}\\
&=(-1)^{m_2(\xi)-u_2-v_2}\sum_{i=0}^{m_2(\xi)-u_2-v_2}\binom{u_1+v_2-1-m_2(\xi)+i}{m_2(\xi)-u_2-v_2-i}\binom{v_1-v_2-1-i}{i}\notag\\
&=(-1)^{m_2(\xi)-u_2-v_2}\sum_{i=0}^{m_2(\xi)-u_2-v_2}\binom{l(\xi)-u_2-v_1+1+i}{m_2(\xi)-u_2-v_2-i}\binom{v_1-v_2-1-i}{i}.\notag
\end{align}
\par The right-hand side of the statement is rewritten with binomial coefficient using Proposition \ref{prop:inv_Kostka_l=2} similarly:
\begin{align*}
&\sum_{\sigma_1,\sigma_2\in\mathfrak S_2}\mathrm{sgn}(\sigma_1)\mathrm{sgn}(\sigma_2)\sum_{\lambda\in C(u,v;\sigma_1,\sigma_2)}K^{-1}_{\xi\lambda'}\\
&=\sum_{\sigma_1,\sigma_2\in\mathfrak S_2}\mathrm{sgn}(\sigma_1)\mathrm{sgn}(\sigma_2)\sum_{\lambda\in C(u,v;\sigma_1,\sigma_2)}(-1)^{m_2(\xi)-\lambda_2}\binom{l(\xi)-\lambda_2}{m_1(\xi)}\\
&=\sum_{\lambda\in C(u,v;\mathrm{id},\mathrm{id})}(-1)^{m_2(\xi)-\lambda_2}\binom{l(\xi)-\lambda_2}{m_1(\xi)}
-\sum_{\lambda\in C(u,v;\mathrm{id},(1,2))}(-1)^{m_2(\xi)-\lambda_2}\binom{l(\xi)-\lambda_2}{m_1(\xi)}\\
&\qquad-\sum_{\lambda\in C(u,v;(1,2),\mathrm{id})}(-1)^{m_2(\xi)-\lambda_2}\binom{l(\xi)-\lambda_2}{m_1(\xi)}
+\sum_{\lambda\in C(u,v;(1,2),(1,2))}(-1)^{m_2(\xi)-\lambda_2}\binom{l(\xi)-\lambda_2}{m_1(\xi)}.
\end{align*}
For $\sigma_2\in\mathfrak S_2$, one obtains that $C(u,v;(1,2),\sigma_2)$ is empty.
It is because $\lambda\in C(u,v;(1,2),\sigma_2)$ satisfies
\[2\lambda_2\geq2u_1+2v_{\sigma_2(2)}\geq u_1+u_2+v_1+v_2>\lambda_1+\lambda_2\]
and there is no such $\lambda$.
Hence the third and the fourth summation is $0$ and
\begin{align*}
&\sum_{\sigma_1,\sigma_2\in\mathfrak S_2}\mathrm{sgn}(\sigma_1)\mathrm{sgn}(\sigma_2)\sum_{\lambda\in C(u,v;\sigma_1,\sigma_2)}K^{-1}_{\xi\lambda'}\\
&=\sum_{\lambda\in C(u,v;\mathrm{id},\mathrm{id})}(-1)^{m_2(\xi)-\lambda_2}\binom{l(\xi)-\lambda_2}{m_1(\xi)}
-\sum_{\lambda\in C(u,v;\mathrm{id},(1,2))}(-1)^{m_2(\xi)-\lambda_2}\binom{l(\xi)-\lambda_2}{m_1(\xi)}\\
&=\sum_{\lambda_2=u_2+v_2}^{u_2+v_1-1}(-1)^{m_2(\xi)-\lambda_2}\binom{l(\xi)-\lambda_2}{m_1(\xi)}\\
&=\sum_{\lambda_2=u_2+v_2}^{u_2+v_1-1}(-1)^{m_2(\xi)-\lambda_2}\binom{l(\xi)-\lambda_2}{m_2(\xi)-\lambda_2}.
\end{align*}
Using Lemma \ref{lem:binom_coeff} with $k=u_2+v_1-\lambda_2-1$ we obtain
\begin{align}\label{inv.kostka.formula2}
&\sum_{\sigma_1,\sigma_2\in\mathfrak S_2}\mathrm{sgn}(\sigma_1)\mathrm{sgn}(\sigma_2)\sum_{\lambda\in C(u,v;\sigma_1,\sigma_2)}K^{-1}_{\xi\lambda'}\\
&=\sum_{\lambda_2=u_2+v_2}^{u_2+v_1-1}(-1)^{m_2(\xi)-\lambda_2}\left(\sum_{i=0}^{m_2(\xi)-\lambda_2}\binom{l(\xi)-u_2-v_1+1+i}{m_2(\xi)-\lambda_2-i}\binom{u_2+v_1-\lambda_2-1-i}{i}\right.\notag\\
&\qquad\left.+\sum_{i=0}^{m_2(\xi)-\lambda_2-1}\binom{l(\xi)-u_2-v_1+1+i}{m_2(\xi)-\lambda_2-i-1}\binom{u_2+v_1-\lambda_2-2-i}{i}\right)\notag\\
&=\sum_{\lambda_2=u_2+v_2}^{u_2+v_1-1}(-1)^{m_2(\xi)-\lambda_2}\sum_{i=0}^{m_2(\xi)-\lambda_2}\binom{l(\xi)-u_2-v_1+1+i}{m_2(\xi)-\lambda_2-i}\binom{u_2+v_1-\lambda_2-1-i}{i}\notag\\
&\qquad+\sum_{\lambda_2=u_2+v_2+1}^{u_2+v_1}(-1)^{m_2(\xi)-\lambda_2+1}\sum_{i=0}^{m_2(\xi)-\lambda_2}\binom{l(\xi)-u_2-v_1+1+i}{m_2(\xi)-\lambda_2-i}\binom{u_2+v_1-\lambda_2-1-i}{i}\notag\\
&=(-1)^{m_2(\xi)-(u_2+v_2)}\sum_{i=0}^{m_2(\xi)-(u_2+v_2)}\binom{l(\xi)-u_2-v_1+1+i}{m_2(\xi)-(u_2+v_2)-i}\binom{u_2+v_1-(u_2+v_2)-1-i}{i}\notag\\
&=(-1)^{m_2(\xi)-u_2-v_2}\sum_{i=0}^{m_2(\xi)-u_2-v_2}\binom{l(\xi)-u_2-v_1+1+i}{m_2(\xi)-u_2-v_2-i}\binom{v_1-v_2-1-i}{i}.\notag
\end{align}
From (\ref{inv.kostka.formula1}) and (\ref{inv.kostka.formula2}) the proposition is proved.
\end{proof}
Now we can prove Theorem \ref{thm:cSS=Q'',n=2}.
\begin{proof}[Proof of Theorem \ref{thm:cSS=Q'',n=2}]
Here we prove the former identity of the theorem.
For $u,v\in A_2$,
\begin{align*}
&\mathrm{sgn}(\tau_u)\mathrm{sgn}(\tau_v)\sum_{\substack{\eta,\zeta\in\widetilde{\mathscr P}^{(2)}\\\eta\cup\zeta=\xi}}K^{-1}_{\eta(\tau_u^{-1}u-\delta)'}K^{-1}_{\zeta(\tau_v^{-1}v-\delta)'}\\
&=\mathrm{sgn}(\tau_u)\mathrm{sgn}(\tau_v)\sum_{\rho_1,\rho_2\in\mathfrak S_2}\mathrm{sgn}(\rho_1)\mathrm{sgn}(\rho_2)\sum_{\lambda\in C(\tau_u^{-1}u,\tau_v^{-1}v;\rho_1,\rho_2)}K^{-1}_{\xi\lambda'}\\
&=\mathrm{sgn}(\tau_u)\mathrm{sgn}(\tau_v)\sum_{\rho_1,\rho_2\in\mathfrak S_2}\mathrm{sgn}(\rho_1)\mathrm{sgn}(\rho_2)\sum_{\lambda\in C(u,v;\tau_u\rho_1,\tau_v\rho_2)}K^{-1}_{\xi\lambda'}\\
&=\sum_{\rho_1,\rho_2\in\mathfrak S_2}\mathrm{sgn}(\tau_u\rho_1)\mathrm{sgn}(\tau_v\rho_2)\sum_{\lambda\in C(u,v;\tau_u\rho_1,\tau_v\rho_2)}K^{-1}_{\xi\lambda'}\\
&=\sum_{\sigma_1,\sigma_2\in\mathfrak S_2}\mathrm{sgn}(\sigma_1)\mathrm{sgn}(\sigma_2)\sum_{\lambda\in C(u,v;\sigma_1,\sigma_2)}K^{-1}_{\xi\lambda'}
\end{align*}
from Proposition \ref{prop:inverse_Kostka}.
Therefore, the left-hand side of the identity are calculated as follows:
\begin{align*}
&\sum_{u,v\in A_2}\mathrm{sgn}(\tau_u)\mathrm{sgn}(\tau_v)q_{2u_1}(x)q_{2u_2+1}(x)q_{2v_1}(x)q_{2v_2+1}(x)\sum_{\substack{\eta,\zeta\in\widetilde{\mathscr P}^{(2)}\\\eta\cup\zeta=\xi}}K^{-1}_{\eta(\tau_u^{-1}u-\delta)'}K^{-1}_{\zeta(\tau_v^{-1}v-\delta)'}\\
&=\sum_{u,v\in A_2}q_{2u_1}(x)q_{2u_2+1}(x)q_{2v_1}(x)q_{2v_2+1}(x)\sum_{\sigma_1,\sigma_2\in\mathfrak S_2}\mathrm{sgn}(\sigma_1)\mathrm{sgn}(\sigma_2)\sum_{\lambda\in C(u,v;\sigma_1,\sigma_2)}K^{-1}_{\xi\lambda'}.
\end{align*}
For $u\in(\mathbb Z_{\geq0})^2\setminus A_2$ and $v\in(\mathbb Z_{\geq0})^2$, $u_1=u_2$ and
\[\sum_{\sigma_1,\sigma_2\in\mathfrak S_2}\mathrm{sgn}(\sigma_1)\mathrm{sgn}(\sigma_2)\sum_{\lambda\in C(u,v;\sigma_1,\sigma_2)}K^{-1}_{\xi\lambda'}=0\]
hold because $C(u,v;\mathrm{id},\sigma_2)=C(u,v;(1,2),\sigma_2)$ for $\sigma_2\in\mathfrak S_2$.
This formula holds for $u\in(\mathbb Z_{\geq0})^2$ and $v\in(\mathbb Z_{\geq0})^2\setminus A_2$ similarly.
Hence
\begin{align*}
&\sum_{u,v\in A_2}\mathrm{sgn}(\tau_u)\mathrm{sgn}(\tau_v)q_{2u_1}(x)q_{2u_2+1}(x)q_{2v_1}(x)q_{2v_2+1}(x)\sum_{\substack{\eta,\zeta\in\widetilde{\mathscr P}^{(2)}\\\eta\cup\zeta=\xi}}K^{-1}_{\eta(\tau_u^{-1}u-\delta)'}K^{-1}_{\zeta(\tau_v^{-1}v-\delta)'}\\
&=\sum_{u,v\in(\mathbb Z_{\geq0})^2}q_{2u_1}(x)q_{2u_2+1}(x)q_{2v_1}(x)q_{2v_2+1}(x)\sum_{\sigma_1,\sigma_2\in\mathfrak S_2}\mathrm{sgn}(\sigma_1)\mathrm{sgn}(\sigma_2)\sum_{\lambda\in C(u,v;\sigma_1,\sigma_2)}K^{-1}_{\xi\lambda'}\\
&=\sum_{u,v\in(\mathbb Z_{\geq0})^2}\sum_{\sigma_1,\sigma_2\in\mathfrak S_2}\sum_{\lambda\in C(u,v;\sigma_1,\sigma_2)}\mathrm{sgn}(\sigma_1)\mathrm{sgn}(\sigma_2)K^{-1}_{\xi\lambda'}q_{2u_1}(x)q_{2u_2+1}(x)q_{2v_1}(x)q_{2v_2+1}(x)\\
&=\sum_{\lambda\in\mathscr P^{(2)}}\sum_{\sigma_1,\sigma_2\in\mathfrak S_2}\sum_{(u,v)\in B(\lambda;\sigma_1,\sigma_2)}\mathrm{sgn}(\sigma_1)\mathrm{sgn}(\sigma_2)K^{-1}_{\xi\lambda'}q_{2u_1}(x)q_{2u_2+1}(x)q_{2v_1}(x)q_{2v_2+1}(x)\\
&=\sum_{\lambda\in\mathscr P^{(2)}}K^{-1}_{\xi\lambda'}\sum_{\sigma_1,\sigma_2\in\mathfrak S_2}\mathrm{sgn}(\sigma_1)\mathrm{sgn}(\sigma_2)\sum_{(u,v)\in B(\lambda;\sigma_1,\sigma_2)}q_{2u_1}(x)q_{2u_2+1}(x)q_{2v_1}(x)q_{2v_2+1}(x)
\end{align*}
and this is equal to
\[\sum_{\lambda\in\mathscr P^{(2)}}2^{-2}K^{-1}_{\xi\lambda'}Q_{2\lambda+2\Delta}(x,x)\]
from Proposition \ref{prop:expression_of_Q}.
It completes the proof of the former identity, and the latter one can be proved in a similar way.
\end{proof}
\section*{Acknowledgments}
The author is grateful to Hiro-Fumi~Yamada for suggesting the topic treated in this paper.
\bibliography{reference}
\bibliographystyle{amsplain}
\end{document}